\documentclass{article}

\PassOptionsToPackage{numbers, sort&compress}{natbib}

 \usepackage[preprint]{neurips_2020}

\usepackage[utf8]{inputenc} 
\usepackage[T1]{fontenc}    
\usepackage{url}            
\usepackage{booktabs}       
\usepackage{amsfonts}       
\usepackage{nicefrac}       
\usepackage{microtype}      

\usepackage{enumitem}

\usepackage{amsmath,amsthm,amsfonts}
\usepackage{xcolor}
\usepackage{booktabs}
\usepackage{mathtools}
\usepackage{algorithm}
\usepackage{algorithmic}
\usepackage{subcaption}
\usepackage{graphicx}
\usepackage[outdir=./]{epstopdf}

\usepackage{hyperref}
\definecolor{darkblue}{rgb}{0.0,0.0,0.65}
\definecolor{darkred}{rgb}{0.65,0.0,0.0}
\hypersetup{breaklinks=true,
   colorlinks = true,
   citecolor  = darkblue,
   linkcolor  = darkred,
   filecolor  = darkblue,
   urlcolor   = darkblue}

\graphicspath{{sections/figures/}}

\newcommand{\defeq}{\mathrel{\mathop:}=}

\newcommand{\mat}[1]{\ensuremath{\mathbf{#1}}}

\newcommand{\N}{\mathbb{N}}
\newcommand{\R}{\mathbb{R}}

\newcommand{\G}{\mathcal{G}}

\newcommand{\W}{\mat{W}}
\newcommand{\X}{\mat{X}}
\newcommand{\bfGamma}{\mat{\Gamma}}
\newcommand{\bfLambda}{\mat{\Lambda}}

\DeclareMathOperator*{\argmin}{arg\,min} 

\newlength\myindent
\newcommand{\our}{IDEAL}
\newcommand{\mixing}{mixing }

\definecolor{darkblue}{rgb}{0,0,0.6}

\newtheorem{theorem}{Theorem}
\newtheorem{lemma}[theorem]{Lemma}
\newtheorem{corollary}[theorem]{Corollary}
\newtheorem{remark}[theorem]{Remark}

\newtheorem{assumption}{Assumption}
\newtheorem{proposition}[theorem]{Proposition}

\title{\our: Inexact DEcentralized Accelerated Augmented Lagrangian Method}

\author{%
  Yossi Arjevani \\
  NYU\\
  \texttt{yossia@nyu.edu} \\
  \And
  Joan Bruna \\
  NYU\\
  \texttt{bruna@cims.nyu.edu} \\  
  \And
  Bugra Can \\
  Rutgers University\\
  \texttt{bc600@scarletmail.rutgers.edu} \\
  \And
  Mert G\"urb\"uzbalaban \\
  Rutgers University\\
  \texttt{mg1366@rutgers.edu} \\ 
  \And
  Stefanie Jegelka \\
  MIT\\
  \texttt{stefje@csail.mit.edu} \\ 
  \And
  Hongzhou Lin \\
  MIT\\
  \texttt{hongzhou@mit.edu} 
}

\begin{document}

\maketitle

\begin{abstract}

We introduce a framework for designing primal methods under the decentralized optimization setting where local functions are smooth and strongly convex. Our approach consists of approximately solving a sequence of sub-problems induced by the accelerated augmented Lagrangian method, thereby providing a systematic way for deriving several well-known decentralized algorithms including EXTRA~\cite{shi2015extra} and SSDA~\cite{scaman2017optimal}. When coupled with 
accelerated gradient descent, our framework yields a novel primal algorithm whose convergence rate is optimal and matched by recently derived lower bounds. We provide experimental results that demonstrate the effectiveness of the proposed algorithm on highly ill-conditioned problems.



\end{abstract}


\section{Introduction}
Due to their rapidly increasing size, modern datasets are typically collected, stored and manipulated in a distributed manner. This, together with strict privacy requirements, has created a large demand for efficient solvers for the decentralized setting in which models are trained locally at each agent, and only local parameter vectors are shared. This approach has become particularly appealing for applications such as edge computing~\cite{shi2016edge,mao2017survey}, cooperative multi-agent learning~\cite{bernstein2002complexity, panait2005cooperative} and federated learning~\cite{mcmahan2017communication,shokri2015privacy}.
Clearly, the nature of the decentralized setting prevents a global synchronization, as only communication within the neighboring machines is allowed. The goal is  then to  arrive at a consensus on all local agents with a model that performs as well as in the centralized setting. 


Arguably, the simplest approach for addressing decentralized settings is to adapt the vanilla gradient descent method to the underlying network architecture  \cite{xiao2007distributed, nedic2009distributed, duchi2011dual, jakovetic2014fast}. To this end, the connections between the agents are modeled through a \mixing matrix, which dictates how agents average over their neighbors' parameter vectors. 
Thus, the \mixing matrix serves as a communication oracle which determines how information propagates throughout the network. 
Perhaps surprisingly, when the stepsizes are constant, simply averaging over the local iterates via the \mixing matrix only converges to a neighborhood of the optimum~\cite{yuan2016convergence,shi2015extra}.
A recent line of works
\cite{shi2014linear, jakovetic2014linear, shi2015extra, qu2017harnessing, nedic2017achieving, nedic2017geometrically} 
proposed a number of alternative methods that linearly converge  to the global minimum. 


The overall complexity of solving decentralized optimization problems is typically determined by two factors: (i) the condition number of the objective function $\kappa_f$, which measures the `hardness' of solving the underlying optimization problem, and (ii) the condition number of the mixing matrix $\kappa_W$, which quantifies the severity of information `bottlenecks' present in the network. Lower complexity bounds recently derived for distributed settings \cite{arjevani2015communication,scaman2017optimal,woodworth2018graph,arjevani2018tight} show that one cannot expect to have a better dependence on the condition numbers than  $\sqrt{\kappa_f}$ and  $\sqrt{\kappa_W}$. 
Notably, despite the considerable recent progress, none of the methods mentioned above is able to achieve accelerated rates, that is, a square root dependence for both ${\kappa_f}$ and ${\kappa_W}$---simultaneously.

An extensive effort has been devoted to obtaining acceleration for decentralized algorithms under various settings \cite{scaman2017optimal, scaman2018optimal,li2018sharp, xu2019accelerated, zhang2019achieving,uribe2020dual, hendrikx2020optimal,dvinskikh2019decentralized,fallah2019robust}. When a dual oracle is available, that is, access to the gradients of the dual functions is provided, optimal rates can be attained for smooth and strongly convex  objectives \cite{scaman2017optimal}.  However, 
having access to a dual oracle is a very restrictive assumption, and resorting to a direct `primalization' through inexact approximation of the dual gradients leads to sub-optimal  worst-case theoretical rates \cite{uribe2020dual}. In this work, we propose a novel primal approach that leads to optimal rates in terms of dependency on $\kappa_f$ and $\kappa_W$. 

Our contributions can be summarized as follows.
\vspace{-0.2cm}
\begin{itemize}[leftmargin=.2in]
    \item We introduce a novel framework based on the accelerated augmented Lagrangian method for designing primal decentralized methods. The framework provides a simple and systematic way for deriving several well-known decentralized algorithms \cite{shi2014linear,jakovetic2014linear,shi2015extra}, including EXTRA~\cite{shi2015extra} and SSDA~\cite{scaman2017optimal}, and unifies their convergence analyses. 
    \item  Using accelerated gradient descent as a sub-routine, we derive a novel method
    for smooth and strongly convex local functions
    which achieves optimal \emph{accelerated} rates on both the condition numbers of the problem,  $\kappa_f$ and $\kappa_W$, using \emph{primal} updates, see~Table~\ref{tab:complexity}. 

    \item We perform a large number of experiments, which confirm our theoretical findings, and demonstrate a significant improvement when the objective function is ill-conditioned and $\kappa_f \gg \kappa_W$. 
\end{itemize}

\section{Decentralized Optimization Setting}
We consider $n$ computational agents and  a network graph $\G= (\mathcal{V}, \mathcal{E})$ which defines how the agents are linked. The set of vertices $\mathcal{V} = \{1, \cdots, n \}$ represents the agents and the set of edges $\mathcal{E} \in \mathcal{V} \times \mathcal{V}$ {specifies} {the} connectivity in the network, {i.e., a communication link between agents $i$ and $j$ exists if and only if $(i,j)\in\mathcal{E}$}. Each agent has access to  local information encoded by a loss function $f_i: \R^d \rightarrow \R$. The goal is to minimize the global {objective} over the entire network, 
\begin{equation}\label{eq:global}
	    \min_{x \in \R^d} f(x) \defeq \sum_{i=1}^n f_i(x).
\end{equation} 
In this paper, we assume that the local loss functions $f_i$ are differentiable, $L$-smooth and $\mu$-strongly convex.\footnote{ \noindent $f$ is $L$-smooth if $\nabla f$ is $L$-Lipschitz; $f$ is $\mu$-strongly convex if $f-\frac{\mu}{2}\|x\|^2$ is convex.} {Strong convexity of the component functions $f_i$ implies that} the problem admits a unique solution, which we denote by $x^*$. 

We consider the following {computation and communication models}~\cite{scaman2017optimal}:
\vspace{-0.1cm}
\begin{itemize}[leftmargin=.2in]
    \item {\bf Local computation:} Each agent {is able to compute} the gradients of $f_i$ and the cost of this computation is one unit of time.
    \item {\bf Communication:} Communication is done synchronously, and each agent {can} only exchange information with its neighbors, where $i$ is a neighbor of $j$ if $(i,j) \in \mathcal{E}$. The ratio between the communication cost and computation cost per round is denoted by $\tau$.
\end{itemize}
\vspace{-0.1cm}


We further assume that propagation of information is governed by a mixing matrix $W \in \R^{n \times n}$ \citep{nedic2009distributed, yuan2016convergence,scaman2017optimal}. Specifically, given a local copy of the decision variable $x_i \in \R^d$ at node $i\in[1,n]$, one round of communication provides the following update $x_i \leftarrow  \sum_{i =1}^n W_{ij} x_j$. The following 
standard assumptions regarding the mixing matrix \cite{scaman2017optimal} are made throughout the paper.
\begin{assumption}\label{ass:W}
The \mixing matrix $W$ satisfies the following:
\begin{enumerate}[leftmargin=.2in]
    \item {\bf Symmetry:} $W = W^T$.
    \item {\bf Positiveness:} $W$ is positive semi-definite.
    \item {\bf Decentralized property}: If $(i,j) \notin \mathcal{E}$ and $i \neq j$, then $W_{ij} = W_{ji} = 0$. 
    \item {\bf Spectrum property:} The kernel of $W$ is given by the vector of all ones $\mbox{Ker}(W)= \R \mathbf{1_n}$. \label{ass:W_spec}
\end{enumerate}
\end{assumption}
A typical choice of the mixing matrix is the (weighted) Laplacian matrix of the graph. Another common choice is to set $W$ as $I- \tilde{W}$ where $\tilde{W}$ is a doubly stochastic matrix \cite{aybat2017decentralized,can2019decentralized,shi2015extra}. By Assumption \ref{ass:W}.\ref{ass:W_spec}, all the eigenvalues of $W$ are strictly positive, except for the smallest one. We let  $\lambda_{\max}(W)$ denote the maximum eigenvalue, and let $\lambda_{\min}^+(W)$ denote the smallest positive eigenvalue.  The ratio between these two quantities plays an important role in quantifying the overall complexity of this problem.


\vspace{0.15cm}
\begin{theorem}[Decentralized lower bound \cite{scaman2017optimal}] 
For any first-order black-box decentralized method, the number of time units required to reach an $\epsilon$-optimal solution for (\ref{eq:global}) is lower bounded by
\begin{equation}\label{eq:lower bound}
    \Omega \left ( \sqrt{\kappa_f} (1 + \tau \sqrt{\kappa_W} )  \log \left (\frac{1}{\epsilon} \right ) \right ),
\end{equation}
where $\kappa_f = L/\mu$ is the condition number of the loss function and $\kappa_W = \lambda_{\max}(W)/ \lambda^+_{\min}(W)$ is the condition number of  the \mixing matrix. 
\end{theorem}
The lower bound decomposes as follows: a) { \bf computation cost}, given by $\sqrt{\kappa_f} \log({1}/{\epsilon})$, and b) {\bf communication cost}, given by $\tau \sqrt{\kappa_f \kappa_W} \log({1}/{\epsilon})$.
The computation cost matches lower bounds for  centralized settings \cite{nesterov2004introductory,arjevani2016iteration}, while the communication cost introduces an additional term which depends on $\kappa_W$ and accounts for the `price' of communication in  decentralized models. It follows that the effective condition number of a given decentralized problem is  $\kappa_W{\kappa_f }$.

Clearly, the choice of the matrix $W$ can strongly affect the optimal attainable performance. For example, $\kappa_W$ can get as large as $n^2$ in the line/cycle graph, or be constant in the complete graph. 
In this paper, we do not focus on optimizing over the choice of $W$ for a given graph $G$; instead, following the approach taken by existing decentralized algorithms, we assume that the graph $G$ and
the mixing matrix W are given and aim to achieve the optimal complexity $\eqref{eq:lower bound}$ for this particular choice of $W$.


 
\section{Related Work and the Dual Formulation}
A standard approach to adress problem \eqref{eq:global} is to express it as a {constrained} optimization problem
\begin{equation}\label{primal}
\min_{\X \in \R^{nd}} 
F(\X):= \frac{1}{n} \sum_{i=1}^n f_i(x_i) \quad \mbox{such that} \quad x_1=x_2= \cdots = x_n \in \mathbb{R}^d\,, \tag{P}
\end{equation}
where $\X = [x_1;x_2; \cdots x_n] \in \R^{nd}$ is a concatenation of the vectors. 
To lighten the notation, we introduce the global mixing matrix $\W = W \otimes I_d \in \R^{nd \times nd}$, {where $\otimes$ denotes the Kronecker product}, and let $\| \cdot \|_{\W}$  denote the semi-norm induced by $\W$, i.e. $\|\X \|^2_{\W} = \X^T \W \X $. With this notation in hand, we briefly review existing literature on decentralized algorithms.

\paragraph{Decentralized Gradient Descent} The decentralized gradient method \cite{nedic2009distributed,yuan2016convergence} has the update rule
\begin{equation}
  \X_{k+1} = \W \X_k - \eta \nabla F(\X_k).  \tag{DGD}
\end{equation}
{However, with constant stepsize, the algorithm does not converge to a global minimum of \eqref{primal}, but rather to a neighborhood of the solution~\citep{yuan2016convergence}.} A decreasing stepsize schedule may be used to ensure convergence, but this yields a sublinear convergence rate, even in the strongly convex case.     


\begin{algorithm}[tb]
	\caption{Decentralized Augmented Lagrangian framework} \label{algo:AL}
	$\,\,$ {\bf Input:} mixing matrix $W$, regularization parameter $\rho$, stepsize $\eta$. \\
	\vspace{-0.4cm}
	\begin{algorithmic}[1]
		\FOR{$k = 1, 2, ..., K$} 
		\STATE $\X_{k} = \argmin  \left \{ P_k(\X) := F(\X) + \mathbf{\Lambda}_k^T \X +  \frac{\rho}{2}  \| \X \|^2_{\W}   \right \} $.
		  \STATE $\mathbf{\Lambda}_{k+1}  = \mathbf{\Lambda}_k + \eta \, \W \X_{k}$.
		\ENDFOR
	\end{algorithmic}
\end{algorithm}

\paragraph{Linearly convergent primal algorithms} 
By and large,  recent methods that achieve linear convergence in the strongly convex case \cite{shi2014linear, jakovetic2014linear, shi2015extra, qu2017harnessing, nedic2017achieving, nedic2017geometrically,sun2019convergence} can be shown to follow a general framework based on the augmented Lagrangian method, see Algorithm~\ref{algo:AL}; The main difference lies in how subproblems $P_k$ are solved.   \citet{shi2014linear} apply an alternating directions method; in  \cite{shi2015extra}, the EXTRA algorithm takes a single gradient descent step to solve~$P_k$, see Appendix~B for details. \citet{jakovetic2014linear} use multi-step algorithms such as Jacobi/Gauss-Seidel methods. To the best of our knowledge, the complexity of these  algorithms is not better than  $O \left ((1+\tau) {\kappa_f \kappa_W}  \log (\frac{1}{\epsilon}) \right )$, in other words, they are non-accelerated. 
The recently proposed algorithm APM-C  \cite{li2018sharp} enjoys a square root dependence on $\kappa_f$ and $\kappa_W$, but incurs an additional $\log(1/\epsilon)$ factor compared to the optimal attainable rate.

\paragraph{Optimal method based on the dual formulation} 
By Assumption \ref{ass:W}.\ref{ass:W_spec}, the constraint $x_1=x_2= \cdots = x_n$ is equivalent to the identity $\W \cdot \X = 0 $, which is again equivalent to $\sqrt{\W} \cdot \X = 0$. Hence, the dual formulation of  \eqref{primal} is given by
\begin{equation}\label{dual}
    \max_{\mathbf{\Lambda} \in \mathbb{R}^{dn}} -F^*(- \sqrt{\W} \mathbf{\Lambda}).  \tag{D}
\end{equation} 
Since the primal function is convex and the constraints are linear, we can use strong duality and address the dual problem instead of the primal one. Using this approach, ~\cite{scaman2017optimal} proposed a dual method with optimal accelerated rates, using Nesterov's accelerated gradient method for the dual problem (\ref{dual}). As mentioned earlier, the main drawback of this method is that it requires access to the gradient of the dual function which, unless the primal function has a relatively simple structure, is not available. 
One may apply a first-order method to approximate the dual gradients inexactly at the expense of an additional $\sqrt{\kappa_f}$ factor in the computation cost~\cite{uribe2020dual}, but this woul make the algorithm no longer optimal. This indicates that achieving optimal rates when using primal updates is a rather challenging task in the decentralized setting. In the following sections, we provide a generic framework which allows us to derive a primal decentralized method with optimal complexity guarantees. 




\section{An Inexact Accelerated Augmented Lagrangian framework}

\begin{algorithm}[tb]
	\caption{Accelerated Decentralized Augmented Lagrangian framework \label{algo:acc-AL}} 
		    {\bf Input:} mixing matrix $W$,  regularization parameter $\rho$, stepsize $\eta$, extrapolation parameters $\{\beta_k\}_{k \in \N}$ \\
	\vspace{-0.4cm}
	\begin{algorithmic}[1]
	    \STATE Initialize dual variables $\mathbf{\Lambda}_1 =\mathbf{\Omega}_1 =\mathbf{0} \in \R^{nd}$.
		\FOR{$k = 1, 2, ..., K$} 
		\STATE $\X_{k} = \argmin  \left \{ P_k(\X) := F(\X) + \mathbf{\Omega}_k^T \X +  \frac{\rho}{2}  \| \X \|^2_{\W}   \right \} $.
		  \STATE $\mathbf{\Lambda}_{k+1}  = \mathbf{\Omega}_k + \eta \W \X_{k}$
		  \STATE $\mathbf{\Omega}_{k+1} =  \mathbf{\Lambda}_{k+1} + \beta_{k+1} (\mathbf{\Lambda}_{k+1} - \mathbf{\Lambda}_k) $
		\ENDFOR
	\end{algorithmic}
	{\bf Output: $\X_K$.} 
\end{algorithm}
\begin{algorithm}[tb]
	\caption{\our : Inexact Acc-Decentralized Augmented Lagrangian framework \label{algo:inexact acc-AL}} 
	{\bf Additional Input:} A first-order optimization algorithm $\mathcal{A}$ \\
	Apply $\mathcal{A}$ to solve the subproblem $P_k$ warm starting at $\X_{k-1}$ to find an approximate solution
    \[\X_{k} \approx \argmin  \left \{ P_k(\X) := F(\X) + \mathbf{\Omega}_k^T \X +  \frac{\rho}{2}  \| \X \|^2_{\W}   \right \} ,\]
    \begin{description}
    \item[Option I:] stop the algorithm when $ \| \X_k - \X_k^*\|^2 \le \epsilon_k$, where $\X_k^*$ is the unique minimizer of $P_k$.
    \item[Option II:] stop the algorithm after a prefixed number of iterations $T_k$.
    \end{description}
\end{algorithm}


In this section, we introduce our inexact accelerated Augmented Lagrangian framework, and show how to combine it with Nesterov's acceleration. To ease the presentation, we first describe a conceptual algorithm, Algorithm \ref{algo:acc-AL}, where subproblems are solved exactly, and only then introduce inexact inner-solvers. 

Similarly to Nesterov's accelerated gradient method, we use  an extrapolation step for the dual variable~$\mathbf{\Lambda}_{k}$. The component $\W \X_k$ in line~4 of Algorithm~\ref{algo:acc-AL} is the negative gradient of the Moreau-envelope\footnote{A proper definition of the Moreau-envelope is given in \cite{rockafellar2009variational}, readers that are not familiar with this concept could take it as an implicit function which shares the same optimum as the original function.} of the dual function. Hence our algorithm is equivalent to applying Nesterov's method on the Moreau-envelope of the dual function, or equivalently, an accelerated dual proximal point algorithm. This renders the optimal dual method proposed in~\cite{scaman2017optimal} as a special case of our algorithmic framework (with $\rho$ set to 0). 


While Algorithm~\ref{algo:acc-AL} is conceptually plausible, it requires an exact solution of the Augmented Lagrangian problems, which can be too expensive in practice. To address this issue, we introduce an inexact version, shown in Algorithm~\ref{algo:inexact acc-AL}, where the $k$-th subproblem $P_k$ is solved up to a predefined accuracy~$\epsilon_k$. The choice of $\epsilon_k$ is rather subtle. On the one hand, choosing a large $\epsilon_k$ may result in a non-converging algorithm. On the other hand,  choosing a small $\epsilon_k$ can be exceedingly expensive as the optimal solution of the subproblem $\X_k^*$ is not the global optimum~$\X^*$.  Intuitively, $\epsilon_k$ should be chosen to be of the same order of magnitude as $\| \X_k^* -\X^*\|$, leading  to the following result.


\begin{theorem}\label{thm:main}
Consider the sequence of primal variables $(\X_k)_{k \in \N}$ generated by Algorithm~\ref{algo:inexact acc-AL} with the subproblem $P_k$ solved up to $\epsilon_k$ accuracy in Option~I. With parameters set to
\begin{equation}
    \beta_k = \frac{\sqrt{L_\rho} - \sqrt{\mu_\rho}}{ \sqrt{L_\rho}+\sqrt{\mu_\rho}}, \quad \eta = \frac{1}{L_\rho}, \quad \epsilon_k = \frac{\mu_\rho}{2 \lambda_{\max}(W)} \left (1-\frac{1}{2}\sqrt{\frac{\mu_\rho}{L_\rho}} \right )^k \Delta_{dual},
\end{equation}
where $L_\rho = \frac{\lambda_{\max}(W)}{\mu + \rho \lambda_{\max}(W) }$, $\mu_\rho = \frac{\lambda_{\min}^+(W)}{L + \rho \lambda_{\min}^+(W)}$ and $\Delta_{dual}$ is the initial dual function gap, we obtain 
\begin{equation}
     \| \X_{k} - \X^* \|^2 \le C_\rho \left ( 1- \frac{1}{2} \sqrt{\frac{\mu_\rho}{L_\rho}}\right )^{k} \Delta_{dual},
\end{equation}
where $\X^*=\mathbf{1}_n \otimes x^*$ and $C_\rho = 258 \frac{L_\rho \lambda_{\max}(W) }{\mu^2 \mu_\rho^2}$.
\end{theorem}
\begin{corollary}
The number of subproblems $P_k$ to achieve $\| \X_k - \X^* \|^2 \le \epsilon$ in \our~is bounded by 
\begin{equation}\label{eq: outer iteration}
    K= O\left( \sqrt{\frac{L_\rho}{\mu_\rho}} \log \left(\frac{C_\rho \Delta_{dual}}{\epsilon} \right )\right).
\end{equation}
\end{corollary}
We remark that inexact accelerated Augmented Lagrangian methods have been previously analyzed  under different assumptions \cite{nedelcu2014computational, kang2015inexact,yan2020bregman}. The main difference is that here, we are able to establish a linear convergence rate, whereas existing analyses only yield sublinear  rates. One of the reasons for this discrepancy is that, although $F^*$ is strongly convex, the dual problem (\ref{dual}) is not, as the mixing matrix $\W$ is singular.  The key to obtaining a linear convergence rate is a fine-grained analysis of the dual problem, showing that the dual variables always lie in the subspace where strong convexity holds. The proof of the theorem relies on the equivalence between Augmented Lagrangian methods and the dual proximal point algorithm \cite{rockafellar1976augmented,bertsekas2014constrained}, which can be interpreted as applying an inexact accelerated proximal point algorithm \cite{guler1992new, lin2017catalyst} to the dual problem. A complete convergence analysis is deferred to Section~C in the appendix.  


{Theorem \ref{thm:main} provides} an accelerated convergence rate with respect to the `augmented' condition number {$\kappa_\rho:=L_\rho/\mu_\rho$, as determined by the Augmented Lagrangian parameter~$\rho$ in Algorithm~\ref{algo:inexact acc-AL}. We have the following bounds:}
\begin{equation}\label{eq:aug condition number}
  \underbrace{1}_{\rho = \infty} \le 
  \kappa_\rho= \frac{L+\rho \lambda_{\min}^+(W) }{\mu + \rho \lambda_{\max}(W)} \frac{\lambda_{\max}(W)}{\lambda_{\min}^+(W)}  \le \underbrace{\frac{L}{\mu} \frac{\lambda_{\max}(W)}{\lambda_{\min}^+(W)}}_{\rho =0} = \kappa_f \kappa_W,
\end{equation}
where we observe that the condition number $\kappa_\rho$ is a decreasing function  of the regularization parameter~$\rho$. When $\rho =0$, the maximum value is attained at $\kappa_\rho=\kappa_f \kappa_W$,  the effective condition number of the decentralized problem. As $\rho$ goes to infinity, the augmented condition number $\kappa_\rho$ {goes to 1}. Naively, one may want to take $\rho$ as large as possible to get a fast convergence. However, one must also take into account the complexity of solving the subproblems. Indeed, since $W$ is singular, the additional regularization term in $P_k$ does not improve the strong convexity of the subproblems, yielding an increase in inner {loops} complexity as $\rho$ grows.
Hence, the optimal choice of~$\rho$ requires balancing {the inner and outer complexity in a careful manner}.



To study the inner loop complexity, we introduce a warm-start strategy.
Intuitively, the distance between $\X_{k-1}$ and the $k$-th {solution }$\X_k^*$ {to the subproblem $P_k$} is roughly  {on the} order of $\epsilon_{k-1}$. 
More precisely, we have the following result.
\begin{lemma}\label{lemma:inner}
Given the parameter choice in Theorem~\ref{thm:main}, initializing the subproblem $P_k$ at $\X_{k-1}$ yields,
\[ \| \X_{k-1} - \X_k^*\|^2  \le \frac{8 C_\rho}{\mu_\rho} \epsilon_{k-1}.  \]
\end{lemma}

 \newcommand{\ra}[1]{\renewcommand{\arraystretch}{#1}}
\begin{table*}\centering 
\ra{1.3}
\begin{tabular}{@{}lcrcrc@{}}\toprule
& {GD} & \phantom{abc}& {AGD} &
\phantom{abc} & {SGD}\\ \midrule
 {$T_k$} 
& $\tilde{O}\left ( \frac{L + \rho \lambda_{\max}(W)}{\mu}  \right )$    && $\tilde{O}\left ( \sqrt{\frac{L + \rho \lambda_{\max}(W)}{\mu}}  \right )$  &&  $\tilde{O} \left (  \frac{\sigma^2}{\mu^2 \epsilon_k}  \right)$ \\
 {$\rho$} & $\frac{L}{\lambda_{\max}(W)}$    &&  $\frac{L}{\lambda_{\max}(W)}$  &&   $\frac{L}{\lambda^+_{\min}(W)}$ \\
 {$\displaystyle   \sum_{k=1}^K T_k$} & $\tilde{O}\left ( \kappa_f \sqrt{\kappa_W} \log(\frac{1}{\epsilon})   \right )$    && $\tilde{O}\left (  \sqrt{\kappa_f \kappa_W} \log(\frac{1}{\epsilon})   \right )$  &&  $\tilde{O} \left (\frac{\sigma^2 \kappa_f \kappa_W}{\mu^2 \epsilon} \right)$ \\
\bottomrule
\end{tabular}
\caption{The first row indicates the number of iterations required for different inner solvers to achieve~$\epsilon_k$ accuracy for the $k$-th subproblem $P_k$; the $\tilde{O}$ notation hides logarithmic factors in the parameters $\rho$, $\kappa_f$ and  $\kappa_W$. The second row shows the theoretical choice of the regularization parameter $\rho$. The last row shows the
total number of iterations according to the choice of $\rho$. }\label{tab:inner complexity}
\end{table*}

 Consequently, the ratio between the initial gap at the $k$-th subproblem and the desired gap $\epsilon_k$ is bounded by
 \[ \frac{\| \X_{k-1} - \X_k^*\|^2 }{\epsilon_k} \le  \frac{8 C_\rho}{\mu_\rho} \frac{\epsilon_{k-1}}{\epsilon_k} \le \frac{16 C_\rho}{\mu_\rho} = O(\kappa_f \kappa_W \rho^2),\]
 which is  independent of $k$. In other words, the inner loop solver only needs to decrease the iterate gap by a constant factor for each $P_k$. If the algorithm enjoys a linear convergence rate, a constant number of iteration is sufficient for that. If the algorithm enjoys a sublinear convergence, then the inner loop complexity grows with~$k$. To illustrate the behaviour of different algorithms, we present the inner loop complexity $T_k$ for gradient descent (GD), accelerated gradient descent (AGD) and stochastic gradient descent (SGD) in Table~\ref{tab:inner complexity}. Note that while the inner complexity  of GD and AGD are independent of $k$, the inner complexity for SGD increases geometrically with $k$. Other possible choices for inner solvers are the alternating directions or Jacobi/Gauss-Seidel method, both of which yield accelerated variants for  \cite{shi2014linear} and \cite{jakovetic2014linear}.
 
 
 In fact, the theoretical upper bounds on the inner complexity also provide a more practical way to halt the inner optimization processes (see Option II in Algorithm~\ref{algo:inexact acc-AL}). Indeed, one can  predefine the computational budget for each subproblem, for instance, $100$ iterations of AGD. If this budget exceeds the theoretical inner complexity $T_k$ in Table~\ref{tab:inner complexity}, then the desired accuracy~$\epsilon_k$ is guaranteed to be reached. In particular, we do not need to evaluate the sub-optimality condition, it is automatically satisfied as long as the budget is chosen appropriately.

 

Finally, the global complexity is obtained by summing $\sum_{k=1}^K T_k$, where $K$ is the number of subproblems given in (\ref{eq: outer iteration}). Note that, so far, our analysis applies to any regularization parameter~$\rho$. Since $\sum_{k=1}^K T_k$ is a function of $\rho$, this implies that one can select the parameter $\rho$ such that the overall  complexity is minimized, leading to the choices of $\rho$ described in Table~\ref{tab:inner complexity}.

\paragraph{Two-fold acceleration} In our setting, acceleration seems to occur in two stages (when compared to the non-accelerated  $O\left ( \kappa_f {\kappa_W} \log(\frac{1}{\epsilon})   \right )$ rates in \cite{shi2014linear, jakovetic2014linear, shi2015extra, qu2017harnessing, nedic2017achieving}). First, combining \our~with GD improves the dependence on the condition of the  mixing matrix $\kappa_W$. Secondly, when used as an inner solver, AGD  improves the dependence on  the condition number of the local  functions $\kappa_f$. This suggests that the two phenomena are independent; while one is related to the consensus between the agents, as governed by the mixing matrix, the other one is related to the respective centralized hardness of the optimization problem.

\paragraph{Stochastic oracle} Our framework also subsumes the stochastic setting, where only noisy gradients are available. In this case, since SGD is  sublinear, the required iteration counters $T_k$ for the subproblem must increase inversely proportional to~$\epsilon_k$. Also the stepsize at the $k$-th iteration needs to be decreased accordingly. The overall complexity is now given by $\tilde{O} \left (\frac{\sigma^2 \kappa_f \kappa_W}{\mu^2 \epsilon} \right)$. 
However, in this case, the resulting dependence on the graph condition number can be improved \cite{fallah2019robust}.


\paragraph{Multi-stage variant (M\our)} We remark that the complexity presented in Table~\ref{tab:complexity} is abbreviated, in the sense that it does not distinguish between communication cost and computation cost. To provide a more fine-grained analysis, it suffices to note that performing a gradient step of the subproblem $\nabla P_k(\X)= \nabla F(\X) + \mathbf{\Omega}_k + \rho \W \X  $ requires one local computation to evaluate $\nabla F$,  and one round of communication to obtain $\W \X$. This implies that when  GD/AGD/SGD is combined with \our, the number of local computation rounds is roughly the number of communication rounds, leading to a sub-optimal computation cost, as shown in Table~\ref{tab:complexity}.

To achieve optimal accelerated rates, we enforce multiple communication rounds after one evaluation of $\nabla F$. This is achieved by substituting the regularization metric $\| \cdot \|^2_\W$ with $\| \cdot \|^2_{Q(\W)}$, where $Q$ is a well-chosen polynomial. In this case, the gradient of the subproblem becomes $\nabla P_k(\X)= \nabla F(\X) + \mathbf{\Omega}_k + \rho \,\, Q(\W) \X $, which requires $\deg(Q)$ rounds of communication.


The choice of the polynomial $Q$ relies on Chebyshev acceleration, which is  introduced in~\cite{ scaman2017optimal,lecturenote}.  More concretely, the Chebyshev polynomials are defined by the recursion relation $T_0(x)=1$, $T_1(x) =x$, $T_{j+1}(x) = 2x T_j(x) - T_{j-1}(x)$, and $Q$ is defined by
\begin{equation}\label{eq:Q}
    Q(x) = 1- \frac{T_{j_W}(c(1-x))}{T_{j_W}(c)} \quad \text{ with } \quad j_W = \lfloor \sqrt{\kappa_W} \rfloor, \quad  c= \frac{\kappa_W +1}{\kappa_W-1}.
\end{equation}
Applying this specific choice of $Q$ to the mixing matrix $W$ reduces its condition number by the maximum amount \cite{ scaman2017optimal,lecturenote}, yielding a graph independent bound $\kappa_{Q(W)} = \lambda_{\max}(Q(W))/ \lambda^+_{\min}(Q(W)) \le4$. Moreover, the symmetry, positiveness and spectrum property in Assumption~\ref{ass:W} are maintained by~$Q(W)$.
Even though $Q(W)$ no longer satisfies the decentralized property, it can be implemented using $\lfloor \sqrt{\kappa_W} \rfloor$ rounds of communications with respect to~$W$. The implementation details of the resulting algorithm are similar to Algorithm~\ref{algo:acc-AL}, and follow by substituting the mixing matrix $W$ by~$Q(W)$ (Algorithm~5 in  Appendix~E).

\begin{table*}\centering 
\ra{1.3}
\begin{tabular}{@{}lcrcrc@{}}\toprule
& {$\rho$} & \phantom{abc}& {Computation cost} &
\phantom{abc} & {Communication cost}\\ \midrule
 {SSDA+AGD} 
& 0   && $\tilde{O} \left (\kappa_f \sqrt{\kappa_W} \log(\frac{1}{\epsilon}) \right)$ &&  $O\left ( \tau  \sqrt{\kappa_f  \kappa_W} \log(\frac{1}{\epsilon}) \right)$ \\
{\our+AGD} & $\frac{L}{\lambda_{\max}(W)}$   && $\tilde{O} \left( \sqrt{\kappa_f \kappa_W} \right ) \log(\frac{1}{\epsilon})$ &&  $ \tilde{O} \left (\tau   \sqrt{\kappa_f \kappa_W} \log(\frac{1}{\epsilon}) \right)$ \\
 {MSDA+AGD} 
& 0   &&  $\tilde{O} \left (\kappa_f \log(\frac{1}{\epsilon}) \right)$ &&  $O\left ( \tau   \sqrt{\kappa_f  \kappa_W} \log(\frac{1}{\epsilon}) \right)$ \\
{M\our+AGD} & $\frac{L}{\lambda_{\max}(Q(W))}$   && $\tilde{O} \left( \sqrt{\kappa_f} \log(\frac{1}{\epsilon}) \right )$ &&  $\tilde{O} \left ( \tau \sqrt{\kappa_f \kappa_W} \log(\frac{1}{\epsilon}) \right)$ \\

\bottomrule
\end{tabular}
\caption{ The communication cost of the presented algorithms are all optimal, but the computation cost differs. An additional factor of $\sqrt{\kappa_f}$ is introduced in SSDA/MSDA compared to their original rate in~\cite{scaman2017optimal}, due to the gradient approximation. The optimal computation cost is achieved by combining our multi-stage algorithm MIDEAL with AGD as an inner solver. 
 }\label{tab:complexity}
\end{table*}

\paragraph{Comparison with inexact SSDA/MSDA \cite{scaman2017optimal}} Recall that SSDA/MSDA are special cases of our algorithmic framework
with the degenerate regularization parameter $\rho=0$. 
Therefore, our complexity analysis naturally extends to an inexact anlysis of SSDA/MSDA, as shown in Table~\ref{tab:complexity}. although the resulting communication costs are optimal,  the computation cost is not, due to the additional $\sqrt{\kappa_f}$ factor  introduced by solving the subproblems inexactly. In contrast, our multi-stage framework achieves the optimal computation cost. 
\begin{itemize}[leftmargin=.1in]
    \item {\bf Low communication cost regime: $\tau \sqrt{\kappa_W}<1$}, the computation cost dominates the communication cost, a $\sqrt{\kappa_f}$ improvement is obtained by MIDEAL comparing to MSDA. 
    \item {\bf Ill conditioned regime: $1 < \tau \sqrt{\kappa_W}< \sqrt{\kappa_f}$}, the complexity of MSDA is dominated by the computation cost $\tilde{O} \left (\kappa_f \log(\frac{1}{\epsilon}) \right)$ while the complexity MIDEAL is dominated by the communication cost $\tilde{O} \left (\tau \sqrt{\kappa_f \kappa_W} \log(\frac{1}{\epsilon}) \right)$. The improvement is proportional to the ratio $\sqrt{\kappa_f}/\tau \sqrt{\kappa_W}$.
    \item {\bf High communication cost regime: $\sqrt{\kappa_f}<\tau \sqrt{\kappa_W}$}, the communication cost dominates, and MIDEAL and MSDA are comparable.
\end{itemize}

\section{Experiments}
\setlength\belowcaptionskip{-2ex}
\begin{figure}[t]
\begin{subfigure}{.33\textwidth}
  \centering
    \includegraphics[width=\linewidth]{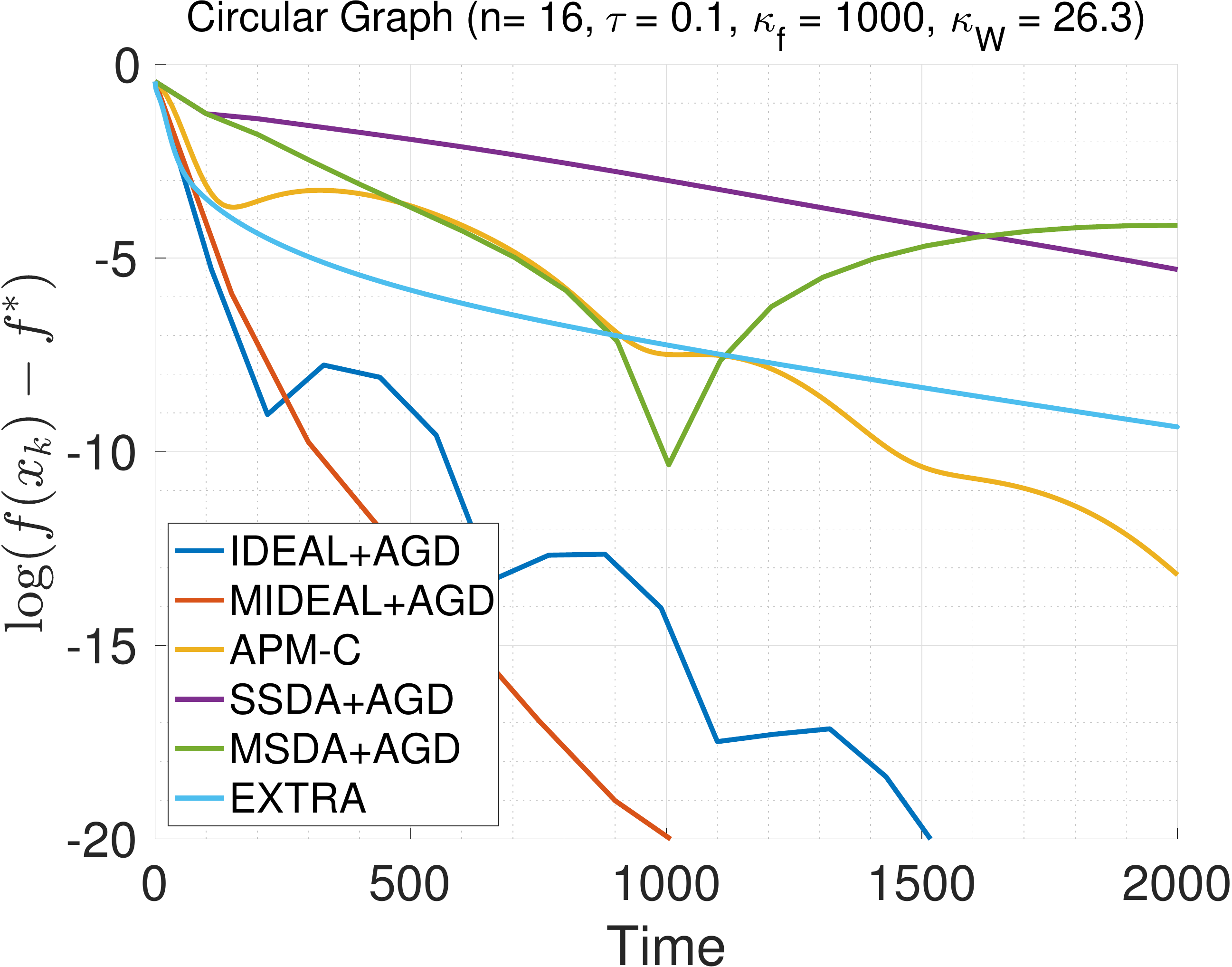}
\end{subfigure}%
\begin{subfigure}{.33\textwidth}
  \centering
    \includegraphics[width=\linewidth]{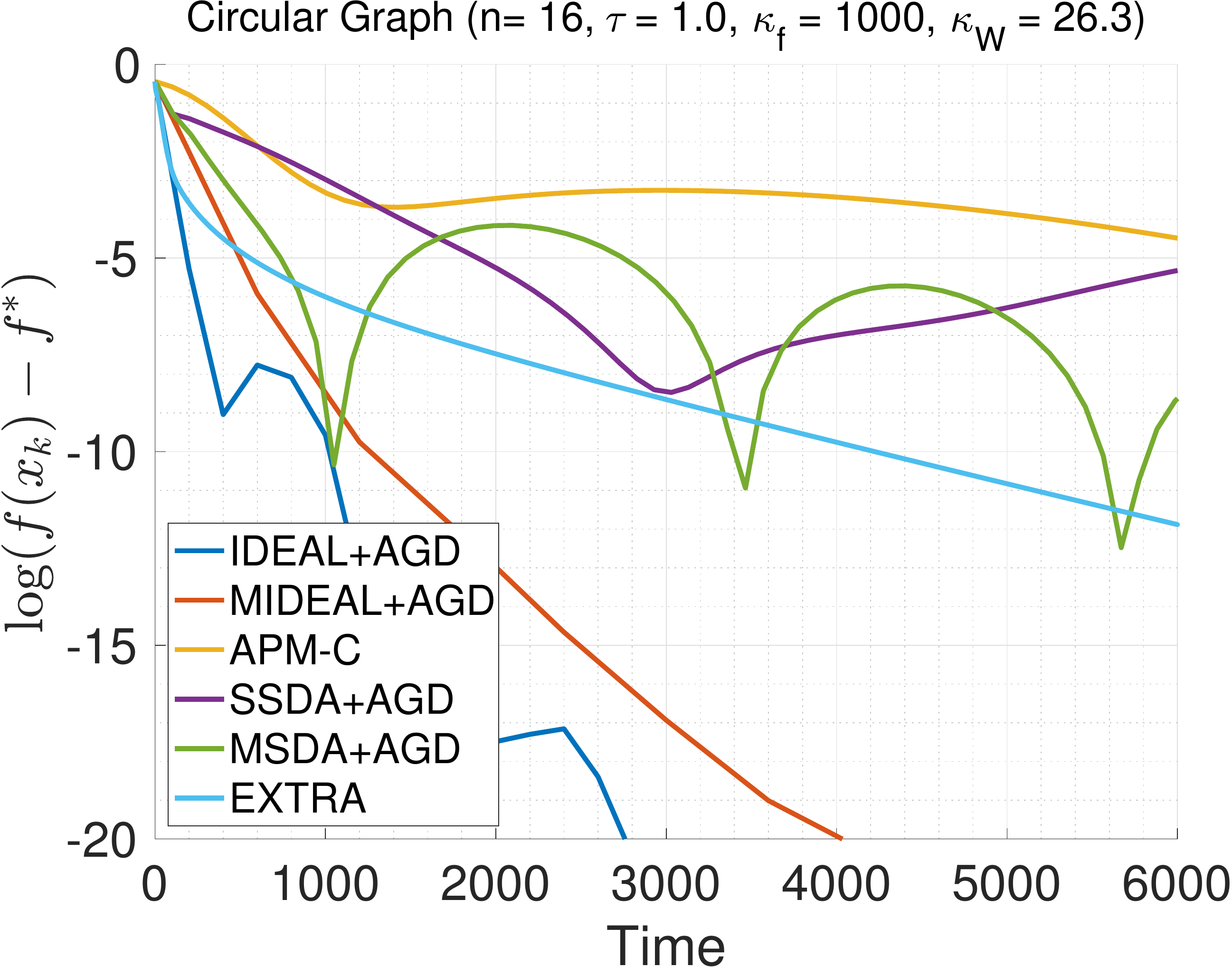}
\end{subfigure}
\begin{subfigure}{.33\textwidth}
  \centering
    \includegraphics[width=\linewidth]{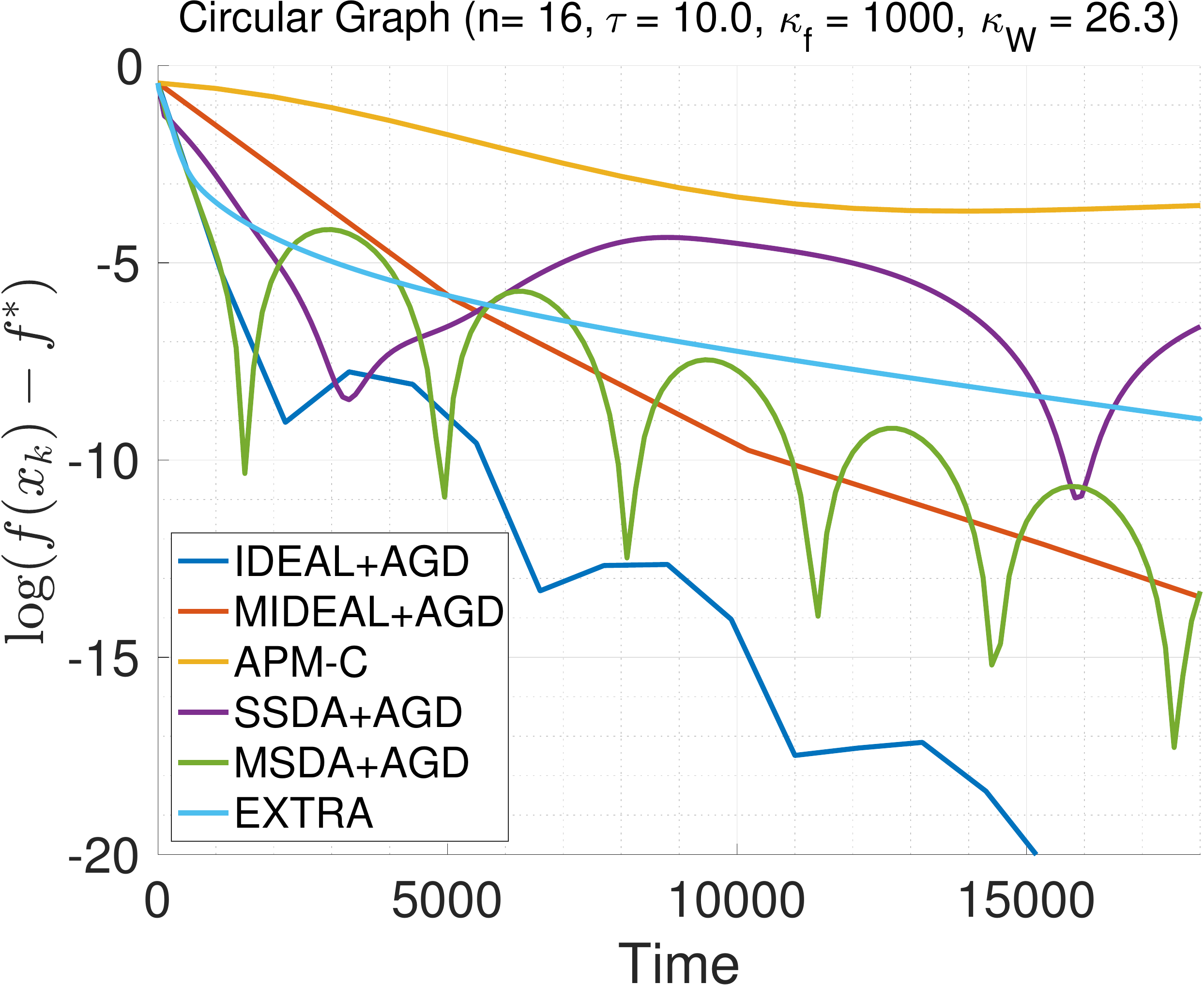}
\end{subfigure}%
\\
\begin{subfigure}{.33\textwidth}
  \centering
    \includegraphics[width=\linewidth]{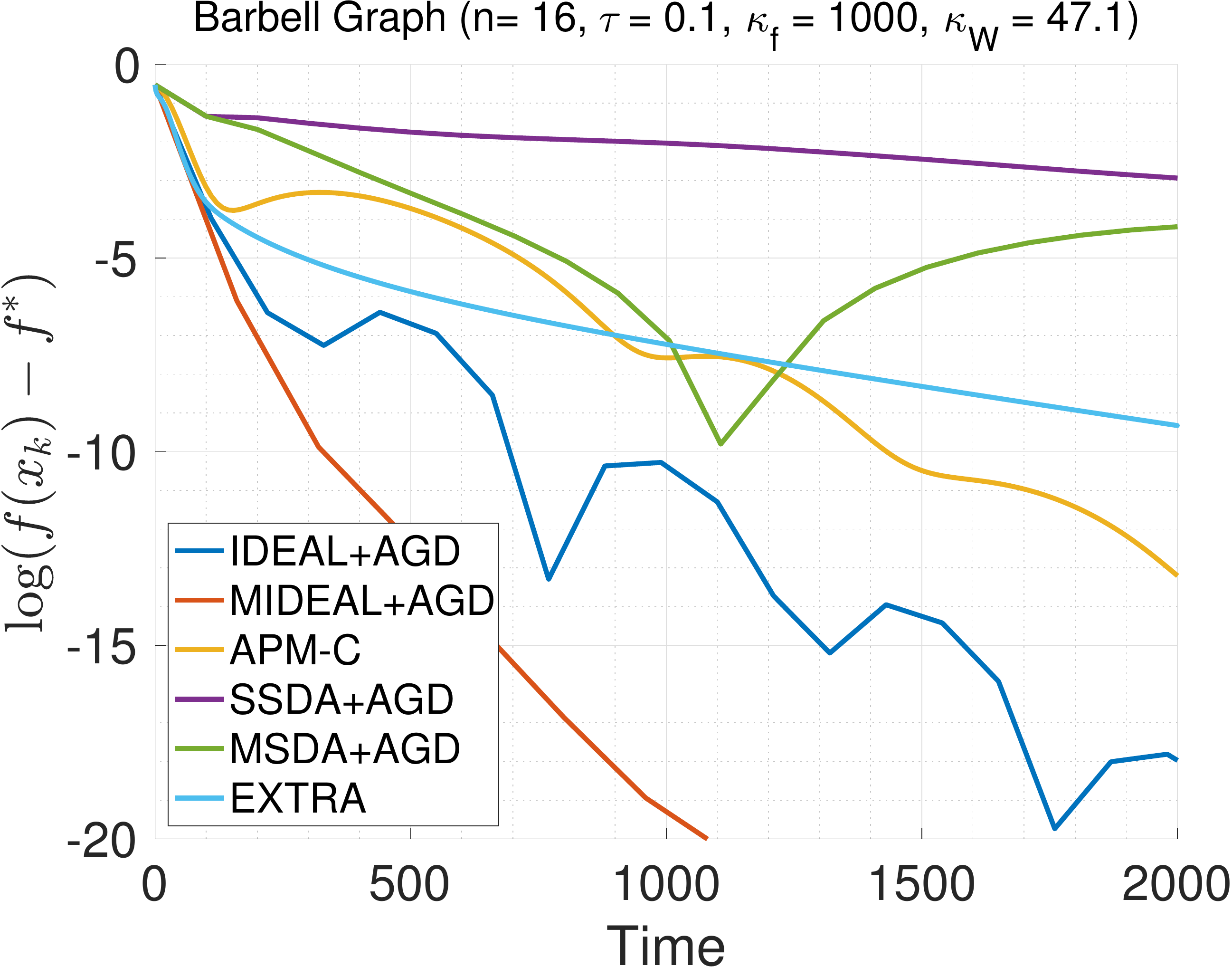}
\end{subfigure}%
\begin{subfigure}{.33\textwidth}
  \centering
    \includegraphics[width=\linewidth]{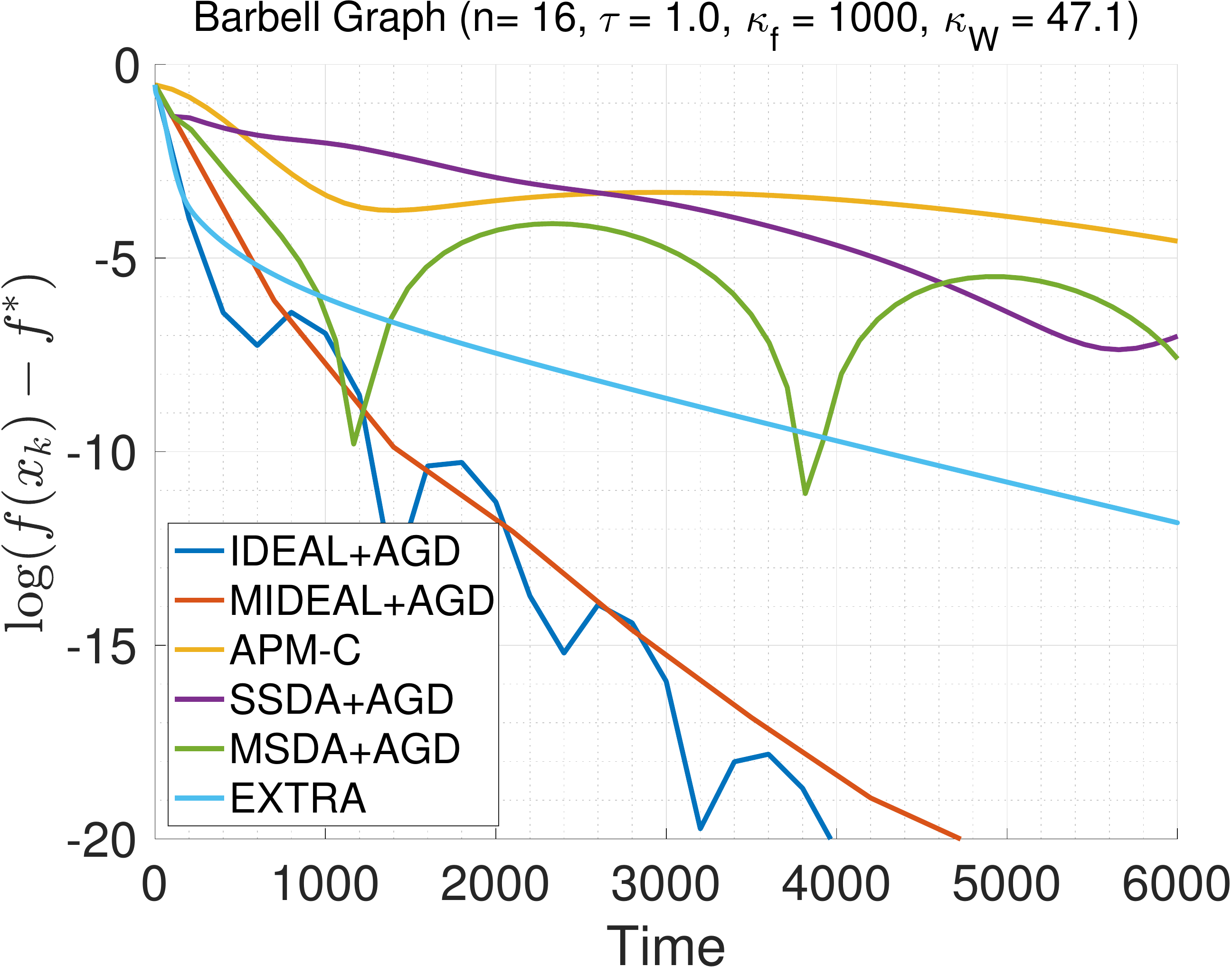}
\end{subfigure}%
\begin{subfigure}{.33\textwidth}
  \centering
    \includegraphics[width=\linewidth]{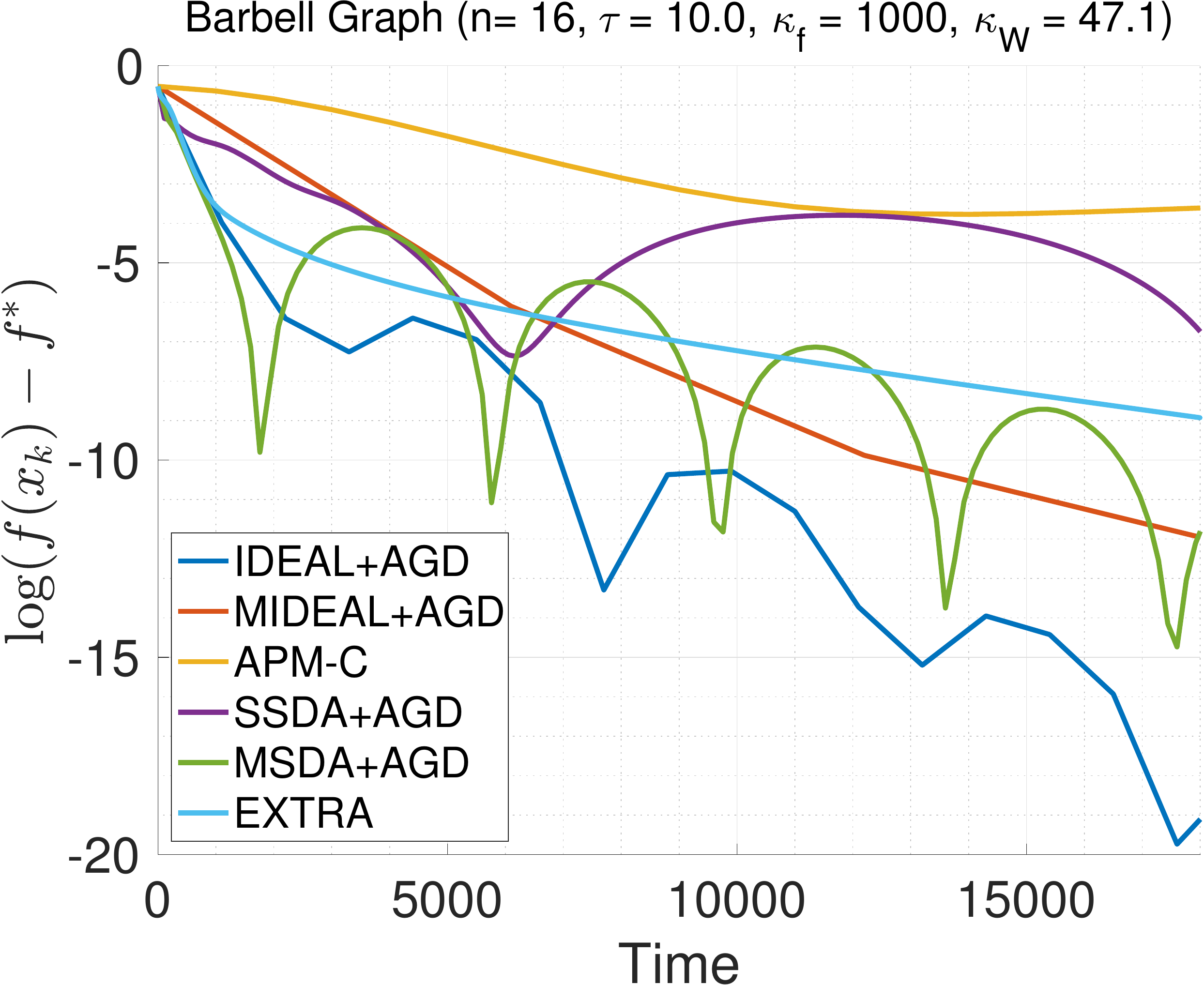}
\end{subfigure}%
\caption{We evaluate the empirical performance of existing state-of-the-art algorithms, where the underlying network is a circular graph (top) and a barbell graph (bottom). We consider the following regimes: low communication cost (left), Ill-condition problems (middle) and High communication cost (right). The x-axis is the  time counter, i.e. the sum of the communication cost and the computation cost; the y-axis is the log scale suboptimality. We observe that our algorithms \our/M\our~are optimal under  
various regimes, validating our theoretical findings.}\label{fig:exps}
\end{figure}
\newcommand{\expnumber}[2]{{#1}\mathrm{e}{#2}}

Having described the \our/M\our{ }algorithms for  decentralized optimization problem \eqref{eq:global}, we now turn to presenting various empirical results which corroborate our theoretical analysis. To facilitate a simple comparison between existing state-of-the-art algorithms, we consider an $\ell_2$-regularized logistic regression task over two classes of the MNIST~\cite{lecun2010mnist} benchmark dataset.  The smoohtness parameter (assuming normalized feature vectors) can be shown to be bounded by $1/4$, which together with a regularization parameter $\mu\approx \expnumber{1}{-3}$, yields a relatively high $\expnumber{1}{3}$-bound on the condition number of the loss function. Further empirical results which demonstrate the robustness of ~\our/M\our\ under wide range of parameter choices are provided in Appendix~G.

We compare the performance of \our/M\our~with the state-of-the-art algorithms EXTRA \cite{shi2015extra}, APM-C \cite{li2018sharp} and the inexact dual method SSDA/MSDA \cite{scaman2017optimal}. We set the inner iteration counter to be $T_k=100$ for all  algorithms, 
and use the theoretical stepsize schedule. The decentralized environment is modelled in a synthetic setting, where the communication time is steady and no latency is encountered. To  demonstrate the effect of the underlying  network architecture, we consider: a) a circular graph, where the agents form a cycle; b) a Barbell graph, where the agents are split into two complete subgraphs, connected by a single bridge (shown in Figure~2 in the appendix). 

As shown in Figure \ref{fig:exps}, our multi-stage algorithm MIDEAL is optimal in the regime where the communication cost $\tau$ is small, and the single-stage variant \our~is optimal when $\tau$ is large. As expected, the inexactness mechanism  significantly slows down the dual method SSDA/MSDA in the low communication cost regime. In contrast, the APM-C algorithm performs reasonably well in the low communication regime, but performs relatively poorly when the communication cost is high.

\section{Conclusions}
We propose a novel framework of decentralized algorithms for smooth and strongly convex objectives. The framework provides a unified viewpoint of several well-known decentralized algorithms and, when instantiated with AGD, 
achieves optimal convergence rates in theory and state-of-the-art performance in practice. We leave further generalization to (non-strongly) convex  and non-smooth objectives to future work.


\newpage

\section*{Acknowledgements}
YA and JB acknowledge support from the Sloan Foundation and Samsung Research.
BC and MG acknowledge support from the grants NSF DMS-1723085 and NSF CCF-1814888.
HL and SJ acknowledge support by The Defense Advanced Research Projects Agency (grant number YFA17
N66001-17-1-4039). The views, opinions, and/or findings contained in this article are those of the
author and should not be interpreted as representing the official views or policies, either expressed or
implied, of the Defense Advanced Research Projects Agency or the Department of Defense.
\section*{Broader impact}

Centralization of data is not always possible because of security and legacy concerns~\cite{voigt2017gdpr}. Our work proposes a new optimization algorithm in the decentralized setting, which can learn a model without revealing the privacy sensitive data. Potential applications include data coming from healthcare, environment, safety, etc,  such as personal medical information~\cite{jochems2016distributed, jochems2017developing}, keyboard input history~\cite{mcmahan2016communication, google2016federated} and beyond.


{
  \bibliographystyle{abbrvnat}
 \bibliography{ref}

\begin{thebibliography}{51}
\providecommand{\natexlab}[1]{#1}
\providecommand{\url}[1]{\texttt{#1}}
\expandafter\ifx\csname urlstyle\endcsname\relax
  \providecommand{\doi}[1]{doi: #1}\else
  \providecommand{\doi}{doi: \begingroup \urlstyle{rm}\Url}\fi

\bibitem[Arjevani and Shamir(2015)]{arjevani2015communication}
Y.~Arjevani and O.~Shamir.
\newblock Communication complexity of distributed convex learning and
  optimization.
\newblock In \emph{Proceedings of Advances in Neural Information Processing
  Systems (NIPS)}, 2015.

\bibitem[Arjevani and Shamir(2016)]{arjevani2016iteration}
Y.~Arjevani and O.~Shamir.
\newblock On the iteration complexity of oblivious first-order optimization
  algorithms.
\newblock In \emph{International Conferences on Machine Learning (ICML)}, 2016.

\bibitem[Arjevani et~al.(2020)Arjevani, Shamir, and Srebro]{arjevani2018tight}
Y.~Arjevani, O.~Shamir, and N.~Srebro.
\newblock A tight convergence analysis for stochastic gradient descent with
  delayed updates.
\newblock In \emph{Proceedings of the 31st International Conference on
  Algorithmic Learning Theory}, volume 117, pages 111--132, 2020.

\bibitem[Auzinger and Melenk(2011)]{lecturenote}
W.~Auzinger and J.~Melenk.
\newblock Iterative solution of large linear systems.
\newblock \emph{Lecture Note}, 2011.

\bibitem[Aybat and G{\"u}rb{\"u}zbalaban(2017)]{aybat2017decentralized}
N.~S. Aybat and M.~G{\"u}rb{\"u}zbalaban.
\newblock Decentralized computation of effective resistances and acceleration
  of consensus algorithms.
\newblock In \emph{2017 IEEE Global Conference on Signal and Information
  Processing (GlobalSIP)}, pages 538--542. IEEE, 2017.

\bibitem[Bernstein et~al.(2002)Bernstein, Givan, Immerman, and
  Zilberstein]{bernstein2002complexity}
D.~S. Bernstein, R.~Givan, N.~Immerman, and S.~Zilberstein.
\newblock The complexity of decentralized control of markov decision processes.
\newblock \emph{Mathematics of operations research}, 27\penalty0 (4):\penalty0
  819--840, 2002.

\bibitem[Bertsekas(2014)]{bertsekas2014constrained}
D.~P. Bertsekas.
\newblock \emph{Constrained optimization and Lagrange multiplier methods}.
\newblock Academic press, 2014.

\bibitem[Can et~al.(2019)Can, Soori, Aybat, Dehvani, and
  G\"urb\"uzbalaban]{can2019decentralized}
B.~Can, S.~Soori, N.~S. Aybat, M.~M. Dehvani, and M.~G\"urb\"uzbalaban.
\newblock Decentralized computation of effective resistances and acceleration
  of distributed optimization algorithms.
\newblock \emph{arXiv preprint arXiv:1907.13110}, 2019.

\bibitem[Duchi et~al.(2011)Duchi, Agarwal, and Wainwright]{duchi2011dual}
J.~C. Duchi, A.~Agarwal, and M.~J. Wainwright.
\newblock Dual averaging for distributed optimization: Convergence analysis and
  network scaling.
\newblock \emph{IEEE Transactions on Automatic control}, 57\penalty0
  (3):\penalty0 592--606, 2011.

\bibitem[Dvinskikh and Gasnikov(2019)]{dvinskikh2019decentralized}
D.~Dvinskikh and A.~Gasnikov.
\newblock Decentralized and parallelized primal and dual accelerated methods
  for stochastic convex programming problems.
\newblock \emph{arXiv preprint arXiv:1904.09015}, 2019.

\bibitem[Fallah et~al.(2019)Fallah, G\"urb\"uzbalaban, Ozdaglar, Simsekli, and
  Zhu]{fallah2019robust}
A.~Fallah, M.~G\"urb\"uzbalaban, A.~Ozdaglar, U.~Simsekli, and L.~Zhu.
\newblock Robust distributed accelerated stochastic gradient methods for
  multi-agent networks.
\newblock \emph{arXiv preprint arXiv:1910.08701}, 2019.

\bibitem[GDPR(2016)]{voigt2017gdpr}
GDPR.
\newblock The eu general data protection regulation (gdpr).
\newblock 2016.

\bibitem[G{\"u}ler(1992)]{guler1992new}
O.~G{\"u}ler.
\newblock New proximal point algorithms for convex minimization.
\newblock \emph{SIAM Journal on Optimization}, 2\penalty0 (4):\penalty0
  649--664, 1992.

\bibitem[Hendrikx et~al.(2020)Hendrikx, Bach, and
  Massoulie]{hendrikx2020optimal}
H.~Hendrikx, F.~Bach, and L.~Massoulie.
\newblock An optimal algorithm for decentralized finite sum optimization, 2020.

\bibitem[Jakoveti{\'c} et~al.(2014{\natexlab{a}})Jakoveti{\'c}, Moura, and
  Xavier]{jakovetic2014linear}
D.~Jakoveti{\'c}, J.~M. Moura, and J.~Xavier.
\newblock Linear convergence rate of a class of distributed augmented
  lagrangian algorithms.
\newblock \emph{IEEE Transactions on Automatic Control}, 60\penalty0
  (4):\penalty0 922--936, 2014{\natexlab{a}}.

\bibitem[Jakoveti{\'c} et~al.(2014{\natexlab{b}})Jakoveti{\'c}, Xavier, and
  Moura]{jakovetic2014fast}
D.~Jakoveti{\'c}, J.~Xavier, and J.~M. Moura.
\newblock Fast distributed gradient methods.
\newblock \emph{IEEE Transactions on Automatic Control}, 59\penalty0
  (5):\penalty0 1131--1146, 2014{\natexlab{b}}.

\bibitem[Jochems et~al.(2016)Jochems, Deist, Van~Soest, Eble, Bulens, Coucke,
  Dries, Lambin, and Dekker]{jochems2016distributed}
A.~Jochems, T.~M. Deist, J.~Van~Soest, M.~Eble, P.~Bulens, P.~Coucke, W.~Dries,
  P.~Lambin, and A.~Dekker.
\newblock Distributed learning: developing a predictive model based on data
  from multiple hospitals without data leaving the hospital--a real life proof
  of concept.
\newblock \emph{Radiotherapy and Oncology}, 121\penalty0 (3):\penalty0
  459--467, 2016.

\bibitem[Jochems et~al.(2017)Jochems, Deist, El~Naqa, Kessler, Mayo, Reeves,
  Jolly, Matuszak, Ten~Haken, van Soest, et~al.]{jochems2017developing}
A.~Jochems, T.~M. Deist, I.~El~Naqa, M.~Kessler, C.~Mayo, J.~Reeves, S.~Jolly,
  M.~Matuszak, R.~Ten~Haken, J.~van Soest, et~al.
\newblock Developing and validating a survival prediction model for nsclc
  patients through distributed learning across 3 countries.
\newblock \emph{International Journal of Radiation Oncology* Biology* Physics},
  99\penalty0 (2):\penalty0 344--352, 2017.

\bibitem[Kang et~al.(2015)Kang, Kang, and Jung]{kang2015inexact}
M.~Kang, M.~Kang, and M.~Jung.
\newblock Inexact accelerated augmented lagrangian methods.
\newblock \emph{Computational Optimization and Applications}, 62\penalty0
  (2):\penalty0 373--404, 2015.

\bibitem[Konečný et~al.(2016)Konečný, McMahan, Yu, Richtarik, Suresh, and
  Bacon]{google2016federated}
J.~Konečný, H.~B. McMahan, F.~X. Yu, P.~Richtarik, A.~T. Suresh, and
  D.~Bacon.
\newblock Federated learning: Strategies for improving communication
  efficiency.
\newblock In \emph{NIPS Workshop on Private Multi-Party Machine Learning},
  2016.
\newblock URL \url{https://arxiv.org/abs/1610.05492}.

\bibitem[LeCun et~al.(2010)LeCun, Cortes, and Burges]{lecun2010mnist}
Y.~LeCun, C.~Cortes, and C.~Burges.
\newblock Mnist handwritten digit database.
\newblock \emph{ATT Labs [Online]}, 2, 2010.
\newblock URL \url{http://yann.lecun.com/exdb/mnist}.

\bibitem[Li et~al.(2018)Li, Fang, Yin, and Lin]{li2018sharp}
H.~Li, C.~Fang, W.~Yin, and Z.~Lin.
\newblock A sharp convergence rate analysis for distributed accelerated
  gradient methods.
\newblock \emph{arXiv preprint arXiv:1810.01053}, 2018.

\bibitem[Lin et~al.(2017)Lin, Mairal, and Harchaoui]{lin2017catalyst}
H.~Lin, J.~Mairal, and Z.~Harchaoui.
\newblock Catalyst acceleration for first-order convex optimization: from
  theory to practice.
\newblock \emph{The Journal of Machine Learning Research}, 18\penalty0
  (1):\penalty0 7854--7907, 2017.

\bibitem[Mairal(2016)]{mairal2016end}
J.~Mairal.
\newblock End-to-end kernel learning with supervised convolutional kernel
  networks.
\newblock In \emph{Proceedings of Advances in Neural Information Processing
  Systems (NIPS)}, 2016.

\bibitem[Mao et~al.(2017)Mao, You, Zhang, Huang, and Letaief]{mao2017survey}
Y.~Mao, C.~You, J.~Zhang, K.~Huang, and K.~B. Letaief.
\newblock A survey on mobile edge computing: The communication perspective.
\newblock \emph{IEEE Communications Surveys \& Tutorials}, 19\penalty0
  (4):\penalty0 2322--2358, 2017.

\bibitem[McMahan et~al.(2017)McMahan, Moore, Ramage, Hampson, and
  y~Arcas]{mcmahan2017communication}
B.~McMahan, E.~Moore, D.~Ramage, S.~Hampson, and B.~A. y~Arcas.
\newblock Communication-efficient learning of deep networks from decentralized
  data.
\newblock In \emph{Artificial Intelligence and Statistics}, pages 1273--1282,
  2017.

\bibitem[McMahan et~al.(2016)McMahan, Moore, Ramage, Hampson,
  et~al.]{mcmahan2016communication}
H.~B. McMahan, E.~Moore, D.~Ramage, S.~Hampson, et~al.
\newblock Communication-efficient learning of deep networks from decentralized
  data.
\newblock \emph{arXiv preprint arXiv:1602.05629}, 2016.

\bibitem[Nedelcu et~al.(2014)Nedelcu, Necoara, and
  Tran-Dinh]{nedelcu2014computational}
V.~Nedelcu, I.~Necoara, and Q.~Tran-Dinh.
\newblock Computational complexity of inexact gradient augmented lagrangian
  methods: application to constrained mpc.
\newblock \emph{SIAM Journal on Control and Optimization}, 52\penalty0
  (5):\penalty0 3109--3134, 2014.

\bibitem[Nedic and Ozdaglar(2009)]{nedic2009distributed}
A.~Nedic and A.~Ozdaglar.
\newblock Distributed subgradient methods for multi-agent optimization.
\newblock \emph{IEEE Transactions on Automatic Control}, 54\penalty0
  (1):\penalty0 48--61, 2009.

\bibitem[Nedic et~al.(2017)Nedic, Olshevsky, and Shi]{nedic2017achieving}
A.~Nedic, A.~Olshevsky, and W.~Shi.
\newblock Achieving geometric convergence for distributed optimization over
  time-varying graphs.
\newblock \emph{SIAM Journal on Optimization}, 27\penalty0 (4):\penalty0
  2597--2633, 2017.

\bibitem[Nedi{\'c} et~al.(2017)Nedi{\'c}, Olshevsky, Shi, and
  Uribe]{nedic2017geometrically}
A.~Nedi{\'c}, A.~Olshevsky, W.~Shi, and C.~A. Uribe.
\newblock Geometrically convergent distributed optimization with uncoordinated
  step-sizes.
\newblock In \emph{2017 American Control Conference (ACC)}, pages 3950--3955.
  IEEE, 2017.

\bibitem[Nesterov(2004)]{nesterov2004introductory}
Y.~Nesterov.
\newblock \emph{Introductory lectures on convex optimization}, volume~87.
\newblock Springer Science \& Business Media, 2004.

\bibitem[Panait and Luke(2005)]{panait2005cooperative}
L.~Panait and S.~Luke.
\newblock Cooperative multi-agent learning: The state of the art.
\newblock \emph{Autonomous agents and multi-agent systems}, 11\penalty0
  (3):\penalty0 387--434, 2005.

\bibitem[Qu and Li(2017)]{qu2017harnessing}
G.~Qu and N.~Li.
\newblock Harnessing smoothness to accelerate distributed optimization.
\newblock \emph{IEEE Transactions on Control of Network Systems}, 5\penalty0
  (3):\penalty0 1245--1260, 2017.

\bibitem[Rockafellar(1976)]{rockafellar1976augmented}
R.~T. Rockafellar.
\newblock Augmented lagrangians and applications of the proximal point
  algorithm in convex programming.
\newblock \emph{Mathematics of operations research}, 1\penalty0 (2):\penalty0
  97--116, 1976.

\bibitem[Rockafellar and Wets(2009)]{rockafellar2009variational}
R.~T. Rockafellar and R.~J.-B. Wets.
\newblock \emph{Variational analysis}, volume 317.
\newblock Springer Science \& Business Media, 2009.

\bibitem[Scaman et~al.(2017)Scaman, Bach, Bubeck, Lee, and
  Massouli{\'e}]{scaman2017optimal}
K.~Scaman, F.~Bach, S.~Bubeck, Y.~T. Lee, and L.~Massouli{\'e}.
\newblock Optimal algorithms for smooth and strongly convex distributed
  optimization in networks.
\newblock In \emph{International Conferences on Machine Learning (ICML)}, 2017.

\bibitem[Scaman et~al.(2018)Scaman, Bach, Bubeck, Massouli{\'e}, and
  Lee]{scaman2018optimal}
K.~Scaman, F.~Bach, S.~Bubeck, L.~Massouli{\'e}, and Y.~T. Lee.
\newblock Optimal algorithms for non-smooth distributed optimization in
  networks.
\newblock In \emph{Proceedings of Advances in Neural Information Processing
  Systems (NIPS)}, 2018.

\bibitem[Schmidt et~al.(2011)Schmidt, Roux, and Bach]{schmidt2011convergence}
M.~Schmidt, N.~L. Roux, and F.~R. Bach.
\newblock Convergence rates of inexact proximal-gradient methods for convex
  optimization.
\newblock In \emph{Proceedings of Advances in Neural Information Processing
  Systems (NIPS)}, 2011.

\bibitem[Shi et~al.(2014)Shi, Ling, Yuan, Wu, and Yin]{shi2014linear}
W.~Shi, Q.~Ling, K.~Yuan, G.~Wu, and W.~Yin.
\newblock On the linear convergence of the {ADMM} in decentralized consensus
  optimization.
\newblock \emph{IEEE Transactions on Signal Processing}, 62\penalty0
  (7):\penalty0 1750--1761, 2014.

\bibitem[Shi et~al.(2015)Shi, Ling, Wu, and Yin]{shi2015extra}
W.~Shi, Q.~Ling, G.~Wu, and W.~Yin.
\newblock Extra: An exact first-order algorithm for decentralized consensus
  optimization.
\newblock \emph{SIAM Journal on Optimization}, 25\penalty0 (2):\penalty0
  944--966, 2015.

\bibitem[Shi et~al.(2016)Shi, Cao, Zhang, Li, and Xu]{shi2016edge}
W.~Shi, J.~Cao, Q.~Zhang, Y.~Li, and L.~Xu.
\newblock Edge computing: Vision and challenges.
\newblock \emph{IEEE internet of things journal}, 3\penalty0 (5):\penalty0
  637--646, 2016.

\bibitem[Shokri and Shmatikov(2015)]{shokri2015privacy}
R.~Shokri and V.~Shmatikov.
\newblock Privacy-preserving deep learning.
\newblock In \emph{Proceedings of the 22nd ACM SIGSAC conference on computer
  and communications security}, pages 1310--1321, 2015.

\bibitem[Sun et~al.(2019)Sun, Daneshmand, and Scutari]{sun2019convergence}
Y.~Sun, A.~Daneshmand, and G.~Scutari.
\newblock Convergence rate of distributed optimization algorithms based on
  gradient tracking.
\newblock \emph{arXiv preprint arXiv:1905.02637}, 2019.

\bibitem[Uribe et~al.(2020)Uribe, Lee, Gasnikov, and Nedi{\'c}]{uribe2020dual}
C.~A. Uribe, S.~Lee, A.~Gasnikov, and A.~Nedi{\'c}.
\newblock A dual approach for optimal algorithms in distributed optimization
  over networks.
\newblock \emph{Optimization Methods and Software}, pages 1--40, 2020.

\bibitem[Woodworth et~al.(2018)Woodworth, Wang, Smith, McMahan, and
  Srebro]{woodworth2018graph}
B.~E. Woodworth, J.~Wang, A.~Smith, B.~McMahan, and N.~Srebro.
\newblock Graph oracle models, lower bounds, and gaps for parallel stochastic
  optimization.
\newblock In \emph{Proceedings of Advances in Neural Information Processing
  Systems (NIPS)}, 2018.

\bibitem[Xiao et~al.(2007)Xiao, Boyd, and Kim]{xiao2007distributed}
L.~Xiao, S.~Boyd, and S.-J. Kim.
\newblock Distributed average consensus with least-mean-square deviation.
\newblock \emph{Journal of parallel and distributed computing}, 67\penalty0
  (1):\penalty0 33--46, 2007.

\bibitem[Xu et~al.(2020)Xu, Tian, Sun, and Scutari]{xu2019accelerated}
J.~Xu, Y.~Tian, Y.~Sun, and G.~Scutari.
\newblock Accelerated primal-dual algorithms for distributed smooth convex
  optimization over networks.
\newblock \emph{International Conference on Artificial Intelligence and
  Statistics (AISTATS)}, 2020.

\bibitem[Yan and He(2020)]{yan2020bregman}
S.~Yan and N.~He.
\newblock Bregman augmented lagrangian and its acceleration, 2020.

\bibitem[Yuan et~al.(2016)Yuan, Ling, and Yin]{yuan2016convergence}
K.~Yuan, Q.~Ling, and W.~Yin.
\newblock On the convergence of decentralized gradient descent.
\newblock \emph{SIAM Journal on Optimization}, 26\penalty0 (3):\penalty0
  1835--1854, 2016.

\bibitem[Zhang et~al.(2019)Zhang, Uribe, Mokhtari, and
  Jadbabaie]{zhang2019achieving}
J.~Zhang, C.~A. Uribe, A.~Mokhtari, and A.~Jadbabaie.
\newblock Achieving acceleration in distributed optimization via direct
  discretization of the heavy-ball ode.
\newblock In \emph{2019 American Control Conference (ACC)}, pages 3408--3413.
  IEEE, 2019.

\end{thebibliography}
}
\newpage

\appendix
\section{Remark on the choice of the mixing matrix}

In the main paper, the mixing matrix $W$ is defined following the convention used in~\cite{scaman2017optimal}, where the kernel of $W$ is the vector of all ones. It is worth noting that the term mixing matrix is also used in the literature to denote a doubly stochastic matrix $W_{DS}$ (see e.g. \cite{shi2014linear, jakovetic2014linear, shi2015extra, qu2017harnessing, nedic2017achieving, nedic2017geometrically,can2019decentralized}). These two approaches are equivalent as
given a doubly stochastic matrix $W_{DS}$, the matrix
\[ I-W_{DS} \text{ is a mixing matrix under Definition~\ref{ass:W}}. \]
In the following discussion, we will use $W_{DS}$ to draw the connection when necessary.

\section{Recovering EXTRA  under the augmented Lagrangian framework} \label{appendix:extra}
The goal of this section is to show that EXTRA algorithm \cite{shi2015extra} is a special case of the non-accelerated Augmented Lagrangian framework in Algorithm~\ref{algo:AL}. 
\begin{proposition}
The EXTRA algorithm is equivalent to applying one step of gradient descent to solve the subproblem in Algorithm~\ref{algo:AL}.  
\end{proposition}
\begin{proof}
Taking a single step of gradient descent in the subproblem $P_k$ in Algorithm~\ref{algo:AL} warm starting at $X_{k-1}$ yields the update
\begin{align}\label{eq:k EXTRA}
    X_k & = X_{k-1} - \alpha (\nabla F(X_{k-1}) + \Lambda_k + \rho W X_{k-1} ). \\
    \Lambda_{k+1} & = \Lambda_k + \eta  W X_{k}. \nonumber
\end{align} 
Using the $(k+1)$-th update, 
\begin{equation}\label{eq:k+1 EXTRA}
     X_{k+1} = X_{k} - \alpha (\nabla F(X_{k}) + \Lambda_{k+1} + \rho W X_{k} ). 
\end{equation}  
and subtracting (\ref{eq:k EXTRA}) from (\ref{eq:k+1 EXTRA}) gives
\[ X_{k+1} = (2 -\alpha (\rho+\eta) W )X_k - (1 - \alpha \rho  W) X_{k-1} - \alpha (\nabla F(X_k) - \nabla F(X_{k-1})).\]
When incorporating with the mixing matrix $W = I -W_{DS}$ and taking $\rho = \eta = \frac{1}{2\alpha}$ gives, 
\[ X_{k+1} = (I+W_{DS})X_k - \left (I+ \frac{W_{DS}}{2} \right ) X_{k-1} - \alpha (\nabla F(X_k) - \nabla F(X_{k-1})), \]
which is the update rule of EXTRA \cite{shi2015extra}.
\end{proof}
\begin{remark}
When expressing the parameters in terms of $\rho$, the inner loop stepsize reads as  $\alpha = \frac{1}{2\rho}$, and the outer-loop stepsize reads as $\eta = \rho$.
\end{remark}

\section{Proof of Theorem 3}\label{app:proof of main thm}
\begin{algorithm}[h]
	\caption{(Unscaled) Accelerated Decentralized Augmented Lagrangian framework} \label{algo:unscaled acc-AL}
		    {\bf Input:} mixing matrix $W$,  regularization parameter $\rho$, stepsize $\eta$, extrapolation parameters $\{\beta_k\}_{k \in \N}$ \\
	\vspace{-0.4cm}
	\begin{algorithmic}[1]
	    \STATE Initialize dual variables $\mathbf{\Lambda}_1 =\mathbf{\Omega}_1 =\mathbf{0} \in \R^{nd}$.
		\FOR{$k = 1, 2, ..., K$} 
		\STATE $\X_{k} \approx \argmin  \left \{ P_k(\X) := F(\X) + (\sqrt{\W} \mathbf{\Omega}_k)^T \X +  \frac{\rho}{2}  \| \X \|^2_{\W}   \right \} $.
		  \STATE $\mathbf{\Lambda}_{k+1}  = \mathbf{\Omega}_k + \eta \sqrt{\W} \X_{k}$
		  \STATE $\mathbf{\Omega}_{k+1} =  \mathbf{\Lambda}_{k+1} + \beta_{k+1} (\mathbf{\Lambda}_{k+1} - \mathbf{\Lambda}_k) $
		\ENDFOR
	\end{algorithmic}
	{\bf Output: $\X_K$.} 
\end{algorithm}

We start by noting that  Algorithm~\ref{algo:acc-AL} is equivalent to the ``unscaled" version of  Algorithm~\ref{algo:unscaled acc-AL}. More specifically, we recover  Algorithm~\ref{algo:acc-AL} by substituting the variables 
\[ \Lambda \leftarrow \sqrt{\W} \Lambda, \quad  \Omega \leftarrow \sqrt{\W} \Omega. \]
The unscaled version is computationally inefficient since it requires the computation of the square root of $W$. This is the reason why we choose to present the scaled version Algorithm~\ref{algo:acc-AL} in the main paper. However, the unscaled version is easier to work with for the analysis. { \bf In the following proof, the variables $\mathbf{\Lambda}$ and $\mathbf{\Omega}$ are referred to  as in the unscaled version  Algorithm~\ref{algo:unscaled acc-AL}.}

The key concept underlying our analysis on is the Moreau-envelope of the dual problem: 
\begin{equation}
    \Phi_\rho(\Lambda) =  \min_{\Gamma \in \R^{nd}} \left \{  F^*(-\sqrt{W} \Gamma ) + \frac{1}{2\rho} \| \Gamma -\Lambda\|^2 \right \}.
\end{equation}
Similarly, we define the associated proximal operator 
\newcommand{\prox}{\operatorname{prox}}
\begin{equation}
    {\prox}_{\Phi_\rho}(\Lambda) =  \argmin_{\Gamma \in \R^{nd}} \left \{  F^*(-\sqrt{W} \Gamma ) + \frac{1}{2\rho} \| \Gamma -\Lambda\|^2 \right \}.
\end{equation}
Note that when the inner problem is strongly convex, the proximal operator is unique (that is, a single-valued operator). The following is a list well known properties of the Moreau-envelope:
\begin{proposition}\label{prop:Moreau}
The Moreau envelope $\Phi_\rho$ enjoys the following properties 
\begin{enumerate}[leftmargin=.2in]
    \item $\Phi_\rho$ is convex and it shares the same optimum as the dual problem (\ref{dual}).
    \item $\Phi_\rho$ is differentiable and the gradient of $\Phi_\rho$ is given by
    \[ \nabla \Phi_\rho(\bfLambda) = \frac{1}{\rho} (\bfLambda - \prox_{\Phi_\rho}(\bfLambda) ) \]
    \item  If $F$ is twice differentiable, then its convex conjugate $F^*$ is also twice differentiable. In this case, $\Phi_\rho$ is also twice differentiable and the Hessian is given by
    \[ \nabla^2 \Phi_\rho (\Lambda) = \frac{1}{\rho} I - \frac{1}{\rho^2} \left [ \frac{1}{\rho} I+ \sqrt{\W} \nabla^2 F^*(- \sqrt{\W} \prox_{\Phi_\rho}(\Lambda)) \sqrt{\W} \right  ]^{-1}. \]
\end{enumerate}
\end{proposition}
\begin{corollary}
The Moreau envelope $\Phi_\rho$ satisfies
\begin{enumerate}[leftmargin=.2in]
    \item $\Phi_\rho$ is $L_\rho$-smooth, where $L_\rho = \frac{\lambda_{\max} (W)}{\mu + \rho \lambda_{\max} (W)} \le \frac{1}{\rho}$.
    \item $\Phi_\rho$ is $\mu_\rho$-strongly convex in the image space of $\sqrt{W}$, where $\mu_\rho = \frac{\lambda_{\min}^+ (W)}{L + \rho \lambda_{\min}^+ (W)}$.
\end{enumerate}
\end{corollary}
\begin{proof}
These properties follow from the expressions for the Hessian of $\Phi_\rho$ and by the fact that $F^*$ is $\frac{1}{\mu}$-smooth and $\frac{1}{L}$ strongly convex.
\end{proof}
In particular, $\Phi_\rho$ is only strongly convex on the image space of $\sqrt{W}$, one of the keys to prove the linear convergence rate is the following lemma.
\begin{lemma}
The variables $\mathbf{\Lambda}_k$ and $\mathbf{\Omega}_k$ in the un-scaled version  Algorithm~\ref{algo:unscaled acc-AL} all lie in the image space of $\sqrt{W}$ for any $k$.
\end{lemma}
\begin{proof}
This can be easily derived by induction according to the update rule in line 4, 5 of Algorithm~\ref{algo:unscaled acc-AL}.
\end{proof}

Similar to the dual Moreau-envelope, we also define the weighted Moreau-envelope on the primal function
\begin{equation}
     \Psi_\rho(\Omega)  = \min_{\X} \left \{ F(\X) + \Omega^T \X + \frac{\rho}{2} \| \X \|_{\W}^2 \right \}
\end{equation}
and its associated proximal operator
\begin{equation}
    \prox_{\Psi_\rho}(\Omega) = \argmin_{\X} \left \{ F(\X) + \Omega^T \X + \frac{\rho}{2} \| \X \|_{\W}^2 \right \} .
\end{equation}
Indeed, this function corresponds exactly to the subproblem solved in the augmented Lagrangian framework (line 3 of Algorithm~\ref{algo:acc-AL}). Similar property holds for $\Psi_\rho$:
\begin{proposition}\label{prop:Moreau primal}
The Moreau envelope $\Psi_\rho$ enjoys the following properties: 
\begin{enumerate}[leftmargin=.2in]
    \item $\Psi_\rho$ is concave.
    \item $\Psi_\rho$ is differentiable and the gradient of $\Psi_\rho$ is given by
    \[ \nabla \Psi_\rho (\Omega) = \prox_{\Psi}(\Omega). \]
    \item  If $F$ is twice differentiable, then $\Psi_\rho$ is also twice differentiable and the Hessian is given by
    \[ \nabla^2 \Psi_\rho (\Omega) = - \left [ \nabla^2 F(\prox_{\Psi}(\Omega)) + \rho W \right ] ^{-1}.  \]
    In particular, $\Psi_\rho$ is $\frac{1}{\mu}$-smooth and $\frac{1}{L+\rho \lambda_{\max}(W)}$ strongly concave. 
\end{enumerate}
\end{proposition}
The dual Moreau-envelope $\Phi_\rho$ and primal Moreau-envelope $\Psi_\rho$ are connected through the following relationship. 
\vspace{0.3cm}
\begin{proposition}\label{prop:gradient MY}
The gradient of the Moreau envelope $\Phi_\rho$ is given by
\begin{equation}\label{eq:gradient of Moreau}
     \nabla \Phi_\rho(\bfLambda) = - \sqrt{\W} \nabla \Psi_\rho (\sqrt{\W} \bfLambda). 
\end{equation}
\end{proposition}
\begin{proof}
To simplify the presentation, let us denote 
 \[ \X(\bfLambda) = \argmin_{\X} \left \{ F(\X) + (\sqrt{\W} \bfLambda)^T \X + \frac{\rho}{2} \| \X \|_{\W}^2 \right \}  = \nabla \Psi_\rho (\sqrt{\W} \bfLambda). \]
From the optimality of $\X(\bfLambda)$, we have 
\[ \nabla F(\X(\bfLambda)) + \sqrt{\W} \bfLambda + \rho \W\X(\bfLambda) =0  \]
From the fact that $\nabla F(x) = y \Leftrightarrow \nabla F^*(y) = x$, we have 
\[ \X(\bfLambda) = \nabla F^*\left ( - \sqrt{\W} \left [ \bfLambda + \rho \sqrt{\W}\X(\bfLambda) \right ] \right ).\]
Let $\bfGamma = \bfLambda + \rho \sqrt{\W}\X(\bfLambda)$, then 
\[ -\sqrt{\W} \nabla F^*(-\sqrt{\W} \bfGamma) + \frac{1}{\rho} (\bfGamma - \bfLambda) = 0. \]
Therefore $\bfGamma$ is the minimizer of the function $ F^*(-\sqrt{W} \Gamma ) + \frac{1}{2\rho} \| \Gamma -\Lambda\|^2 $, namely
\[ \prox_{\Phi_\rho}(\bfLambda) = \bfLambda + \rho \sqrt{\W}\X(\bfLambda) . \]
Then based on the expression for the gradient in Prop~\ref{prop:Moreau}, we obtain the desired equality (\ref{eq:gradient of Moreau}). 
\end{proof}

Proposition \ref{eq:gradient of Moreau} demonstrates that solving the augmented Lagrangian subproblem could be viewed as evaluating the gradient of the Moreau-envelope. Hence applying gradient descent on the Moreau-envelope gives the non-accelerated augmented Lagrangian framework~Algorithm~\ref{algo:AL}. Even more, applying Nesterov's accelerated gradient on the Moreau-envelope $\Phi_\rho$ yields accelerated Augmented Lagrangian Algorithm~\ref{algo:unscaled acc-AL}. In addition, when the subproblems are solved inexactly, this corresponds to an inexact evaluation on the gradient. This interpretation allows us to derive guarantees for the convergence rate of  the dual variables. Before present the the convergence analysis in detail, we formally establish the  connection between the primal solution and the dual solution. 
\begin{lemma}
Let $x^*$ be the optimum of $f$ and define $\X^*=\mathbf{1}_n \otimes x^* \in \R^{nd}$. Then there exists a unique $\bfLambda^* \in Im(\W)$ such that $\bfLambda^*$ is the optimum of the dual problem~(\ref{dual}). Moreover, it satisfies
\[  \nabla F(\X^*) = -\sqrt{\W} \bfLambda^*. \]
\end{lemma}
\begin{proof}
Since $Ker(W) = \R \mathbf{1}_n$, we have 
\[ Ker(\W) = Ker(W\otimes I_d) = Vect(\mathbf{1}_n \otimes \mathbf{e}_i, i=1,\cdots, d), \]
where $\mathbf{e}_i$ is the canonical basis with all entries 0 except the $i$-th equals to 1. By optimality, $\nabla f(x^*) = \sum_{i=1}^n \nabla f_i(x^*)=0$. This implies that $\nabla F(x^*)^T (\mathbf{1}_n \otimes \mathbf{e}_i) = 0$, for all $i =1, \cdots d$. In other words, $\nabla F(X^*)$ is orthogonal to the null space of $\W$, namely $\nabla F(X^*) \in Im(\W)$. Therefore, there exists $\bfLambda$ such that $\nabla F(\X^*) = -\W \bfLambda$. By setting $\bfLambda^* = \sqrt{\W} \bfLambda$, we have $\bfLambda^* \in Im(\W)$ and $\nabla F(\X^*) = -\sqrt{\W} \bfLambda^*$. In particular, since $\nabla F(x) = y \Leftrightarrow \nabla F^*(y) = x$, we have,
\begin{equation}
    \sqrt{\W} \nabla F^*(-\sqrt{\W} \bfLambda^*) = \sqrt{\W} \X^* = 0. 
\end{equation} 
Hence $\bfLambda^*$ is the solution of the dual problem~(\ref{dual}) and it is the unique one lies in the $Im(\W)$. 
\end{proof}
Throughout the rest of the paper, we use $\Lambda^*$ to denote the unique solution as shown in the lemma above. We would like to emphasize that even though $F^*$ is strongly convex, the dual problem~(\ref{dual}) is not strongly convex, because $W$ is singular. Hence, the solution of the dual problem is not unique unless we restrict to the image space of $\W$. To derive the linear convergence rate, we need to show that the dual variable always lies in this subspace where the Moreau-envelope $\Phi_\rho$ is strongly convex.

\vspace{0.3cm}
\begin{theorem}
Consider the sequence of primal variables $(\X_k)_{k \in \N}$ generated by Algorithm~\ref{algo:inexact acc-AL} with the subproblem solved up to $\epsilon_k$ accuracy, i.e. Option~I. Therefore, 
\begin{equation}\label{eq:convergence primal epsilon}
     \| \X_{k+1} - \X^* \|^2 \le 2 \epsilon_{k+1} +  C \left (1- \sqrt{\frac{\mu_\rho}{L_\rho}} \right )^k \left (\sqrt{ \mu_\rho \Delta_{dual}} + A_k \right )^2 
\end{equation}
where $\X^*=\mathbf{1}_n \otimes x^*$, $L_\rho = \frac{\lambda_{\max} (W)}{\mu + \rho \lambda_{\max} (W)}$, $\mu_\rho = \frac{\lambda_{\min}^+ (W)}{L + \rho \lambda_{\min}^+ (W) }$, $C = \frac{2\lambda_{\max} (W)}{\mu^2 \mu_\rho^2}$, $\Delta_{dual}$ is the dual function gap defined by $\Delta_{dual} = F^*(-\sqrt{\W} \bfLambda_1) -F^*(-\sqrt{\W} \bfLambda^*)$ and $A_k=  \sqrt{\lambda_{\max} (W)}\sum_{i=1}^k \sqrt{\epsilon_i} \left( 1 - \sqrt{\frac{\mu_\rho}{L_\rho}} \right )^{-i/2}.$
\end{theorem}
\begin{proof} The proof builds on the concepts developed so far in this section. 
We start by showing that the dual variable $\bfLambda_k$ converges to the dual solution $\bfLambda^*$ in a linear rate. From the interpretation given in Prop~\ref{prop:Moreau} and Prop~\ref{prop:gradient MY}, the sequence $(\Lambda_k)_{k\in \N}$ given in Algorithm~\ref{algo:acc-AL} is equivalent to applying Nesterov's accelerated gradient method on the Moreau-envelope $\Phi_\rho$. In the inexact variant, the inexactness on the solution directly translates to an inexact gradient of $\Phi_\rho$, where the inexactness is given by 
\[ \|e_k\| =  \| \sqrt{\W} (X_k - X_k^*) \| \le \sqrt{\lambda_{\max} (W)} \| X_k - X_k^* \| \le \sqrt{\lambda_{\max} (W) \epsilon_k}. \]
Hence $(\Lambda_k)_{k \in \N}$ in Algorithm~\ref{algo:unscaled acc-AL} is obtained by applying inexact accelerated gradient method on the Moreau-envelope $\Phi_\rho$. Note that by induction $\Lambda_k$ and $\Omega_k$ belong to the image space of $\sqrt{\W}$, in which the dual Moreau-envelope $\Phi_\rho$ is strongly convex. Following the analysis on inexact accelerated gradient method Prop~4 in \cite{schmidt2011convergence}, we have
\begin{equation}
    \frac{\mu_\rho}{2} \| \Lambda_{k+1} - \Lambda^* \|^2 \le \left (1- \sqrt{\frac{\mu_\rho}{L_\rho}} \right )^{k+1} \left (\sqrt{2 \Delta_{\Phi_\rho}} + \sqrt{\frac{2}{\mu_\rho}} A_k \right )^2
\end{equation}
where $\Delta_{\Phi_\rho} = \Phi_\rho(\Lambda_1) - \Phi_\rho^*$ and $A_k$ is the accumulation of the errors given by
\[ A_k= \sum_{i=1}^k \| e_i \| \left( 1 - \sqrt{\frac{\mu_\rho}{L_\rho}} \right )^{-i/2} \le  \sum_{i=1}^k \sqrt{ \lambda_{\max} (W) \epsilon_i} \left( 1 - \sqrt{\frac{\mu_\rho}{L_\rho}} \right )^{-i/2}.\] 

Based on the convergence on the dual variable, we could now derive the convergence on the primal variable. Let $X_{k+1}^*$ be the exact solution of the problem $P_{k+1}$. Then 
\begin{align}\label{eq:xk*}
  \| \X_{k+1}^* -\X^* \|  & = \| \nabla \Psi_\rho (\sqrt{\W} \bfLambda_{k+1}) - \nabla \Psi_\rho (\sqrt{\W} \bfLambda^*) \| \nonumber \\
 & \le   \frac{1}{\mu} \|\sqrt{\W} (\bfLambda_{k+1} -\bfLambda^*) \|  \quad \text{ (From Prop~\ref{prop:Moreau primal}.3)} \nonumber \\
    & \le \frac{\sqrt{\lambda_{\max} (W)}}{\mu} \|\bfLambda_{k+1} -\bfLambda^* \| .
\end{align}
Finally, from triangle inequality 
\begin{align*}
    \| \X_{k+1} -\X^* \|^2 & \le 2\| \X_{k+1} -\X_{k+1}^* \|^2 + 2\| \X_{k+1}^* -\X^* \|^2  \\
    & \le 2 \epsilon_{k+1} +  \frac{2\lambda_{\max}(W)}{\mu^2 \mu_\rho} (1- \sqrt{\kappa_\rho})^{k+1} \left (\sqrt{2 \Delta_{\Phi_\rho}} + \sqrt{\frac{2}{\mu_\rho}} A_k \right )^2.
\end{align*}
The desired inequality follows from reorganizing the constant and the fact that $\Delta_{\Phi_\rho} \le \Delta_{dual}$.
\end{proof}
\begin{proof}[Proof of Theorem~\ref{thm:main}] Plugging in the choice of $\epsilon_k =  \frac{\mu_\rho}{2 \lambda_{\max}(W)} \left (1-\frac{1}{2}\sqrt{\frac{\mu_\rho}{L_\rho}} \right )^k \Delta_{dual}$ in (\ref{eq:convergence primal epsilon}) yields the desired convergence rate. 

\end{proof}

\section{Proof of Lemma~\ref{lemma:inner}}
\begin{lemma}
With the parameter choice as Theorem~\ref{thm:main}, then 
warm starting the $k$-th subproblem $P_k$ at the previous solution $\X_{k-1}$ gives an initial gap
\[ \| \X_{k-1} - \X_k^*\|^2  \le 8\frac{C_\rho}{\mu_\rho} \epsilon_{k-1}.  \]
\end{lemma}
\begin{proof}
From triangle inequality, we have 
\[ \| \X_{k-1} - \X_k^*\|^2 \le 2(\| \X_{k-1} - \X^*\|^2 + \| \X_k^*- \X^* \|^2) \]
The desired inequality follows from the convergence on the primal iterates and (\ref{eq:xk*}), i.e.
\[ \| \X_{k-1} - \X_k^*\|^2 \le \frac{2 C_\rho}{\mu_\rho} \epsilon_{k-1}, \quad \| \X_{k}^* - \X_k^*\|^2  \le \frac{2 C_\rho}{\mu_\rho} \epsilon_{k}.\]
\end{proof}

\section{Multi-stage algorithm: MIDEAL} \label{appendix: mideal}

\begin{algorithm}[h]
	\caption{MIDEAL: Multi-stage Inexact Acc-Decentralized Augmented Lagrangian framework} \label{algo:mideal}
		    {\bf Input:} mixing matrix $W$,  regularization parameter $\rho$, stepsize $\eta$, extrapolation parameters $\{\beta_k\}_{k \in \N}$\\
	\vspace{-0.4cm}
	\begin{algorithmic}[1]
	    \STATE Initialize dual variables $\mathbf{\Lambda}_1 =\mathbf{\Omega}_1 =\mathbf{0} \in \R^{nd}$ and the polynomial $Q$ according to (\ref{eq:Q}).
		\FOR{$k = 1, 2, ..., K$} 
		\STATE $\X_{k} \approx \argmin  \left \{ P_k(\X) := F(\X) + \mathbf{\Omega}_k^T \X +  \frac{\rho}{2}  \| \X \|^2_{Q(\W)}   \right \} $.
		  \STATE $\mathbf{\Lambda}_{k+1}  = \mathbf{\Omega}_k + \eta Q(\W) \X_{k}$
		  \STATE $\mathbf{\Omega}_{k+1} =  \mathbf{\Lambda}_{k+1} + \beta_{k+1} (\mathbf{\Lambda}_{k+1} - \mathbf{\Lambda}_k) $
		\ENDFOR
	\end{algorithmic}
	{\bf Output: $\X_K$.} 
\end{algorithm}

\begin{algorithm}[h]
	\caption{AcceleratedGossip~\cite{scaman2017optimal}} \label{algo:accgossip}
		    {\bf Input:} mixing matrix $W$, vector or matrix $X$.\\
	\vspace{-0.4cm}
	\begin{algorithmic}[1]
	    \STATE Set parameters $\kappa_W = \tfrac{\lambda_{\max}(W)}{ \lambda_{\min}^+(W)}$, $c_2= \tfrac{\kappa_W +1}{\kappa_W -1}$, $c_3 = \tfrac{2}{(\kappa_W +1)\lambda_{\min}^+(W)}$, \# of iterations $J = \lfloor \sqrt{\kappa_W} \rfloor$.
	    \vspace*{-0.3cm}
	    \STATE Initialize coefficients $a_0=1$, $a_1=c_2$, iterates $X_0=X$, $X_1 = c_2(I- c_3 W) X$.
		\FOR{$j = 1, 2, ..., J-1$} 
		\STATE $a_{j+1} = 2c_2a_j - a_{j-1}$
		  \STATE $X_{j+1} = 2c_2(1-c_3 W) X_j - X_{j-1}$
		\ENDFOR
	\end{algorithmic}
	{\bf Output: $X_0- \frac{X_{J}}{a_{J}}$.} 
\end{algorithm}
Intuitively, we simply replace the mixing matrix $W$ by $Q(W)$, resulting in a better graph condition number. However, each evaluation of the new mixing matrix $Q(W)$ requires $deg(Q)$ rounds of communication, given by the AcceleratedGossip~algorithm introduced in \cite{scaman2017optimal}. For completeness of the discussion, we recall this procedure in Algorithm~\ref{algo:accgossip}. In particular, given $W$ and $X$, AcceleratedGossip($W$,$X$) returns $Q(W)X$, based on the communication oracle $W$.

\section{Implementation of Algorithms}

\begin{algorithm}[h]
	\caption{ Implementation: \our+AGD solver }
		    {\bf Input:} number of iterations $K > 0$, gossip matrix $W \in \R^{n \times n}$
	\begin{algorithmic}[1]
        \STATE $\omega_i(0) = \vec{0}$, $\gamma_i(0)=\vec{0}$, $x_i(0)= \overline{x_i}(0) = x_0$ for any $i \in [1,n]$
        \STATE $\kappa_{inner} = \frac{L + \rho\lambda_{\max}(W)}{\mu}$, $\beta_{inner} = \frac{\sqrt{\kappa_{inner}}-1}{\sqrt{\kappa_{inner}}+1}$, $\kappa_{\rho} = \frac{L+\rho \lambda_{\min}^+(W) }{\mu + \rho \lambda_{\max}(W)} \frac{\lambda_{\max}(W)}{\lambda_{\min}^+(W)}$, $\beta_{outer} = \frac{\sqrt{\kappa_{outer}}-1}{\sqrt{\kappa_{outer}}+1}$
		\FOR{$k = 1, 2, ..., K$} 
		\STATE { \color{blue} Inner iteration: Approximately solve the augmented Lagrangian multiplier.}
		\begin{ALC@g}
		\STATE Set $x_{i,k}(0)= y_{i,k}(0)= x_i(k-1)$, $\overline{x_{i,k}}(0) = \overline{y_{i,k}}(0) = \sum_{j \sim i} W_{ij} x_{j,k}(0)$
		    \FOR{$t = 0, 1, ..., T-1$} 
		\STATE $x_{i,k}(t+1) =  y_{i,k}(t)  - \eta (\gamma_i(k) + \nabla f_i(y_{i,k}(t))+ \rho \overline{y_{i,k}}(t))$
		\STATE $y_{i,k}(t+1) = x_{i,k}(t+1) + \beta_{inner} (x_{i,k}(t+1) - x_{i,k}(t) ) $
		\STATE $\overline{y_{i,k}}(t+1) = \sum_{j \sim i} W_{ij} y_{j,k}(t+1)$
		\ENDFOR
				\STATE Set $x_{i}(k) = x_{i,k}(T)$, $ \overline{x_i}(k)= \sum_{j \sim i} W_{ij} x_{j,k}(T) $
		\end{ALC@g}
		\STATE { \color{blue}  Outer iteration: Update the dual variables on each node}
		  \begin{ALC@g}
		  \STATE $\lambda_{i}(k+1)  = \omega_{i}(k) + \rho \overline{x_{i}}(k)$
		  \STATE $\omega_{i}(k+1) =  \lambda_{i}(k+1) + \beta_{outer} (\lambda_{i}(k+1) - \lambda_{i}(k)) $
		  \end{ALC@g}
		\ENDFOR
	\end{algorithmic}
	{\bf Output:} 
\end{algorithm}


\section{Further Experimental Results}\label{appendix:exp}

\begin{figure}[h]
    \centering
    \includegraphics[width=0.8\linewidth]{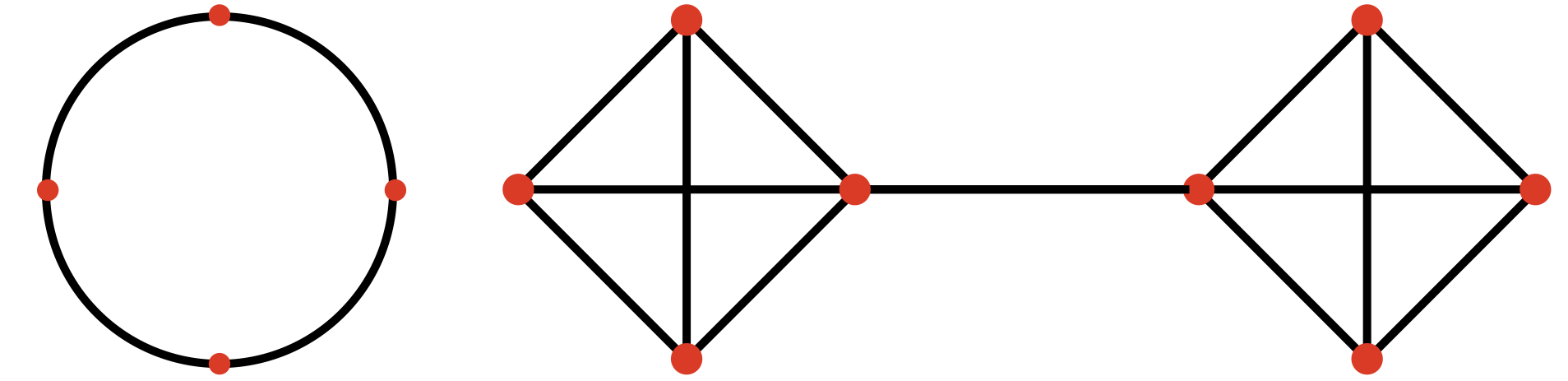}
    \caption{Network Structures: \textbf{Left:}Circular graph with 4 nodes. \textbf{Right:}Barbell graph with 8 nodes.}
    \label{fig: Network-Samples}
\end{figure}

\begin{figure}[t]
\begin{subfigure}{.33\textwidth}
  \centering
    \includegraphics[width=\linewidth]{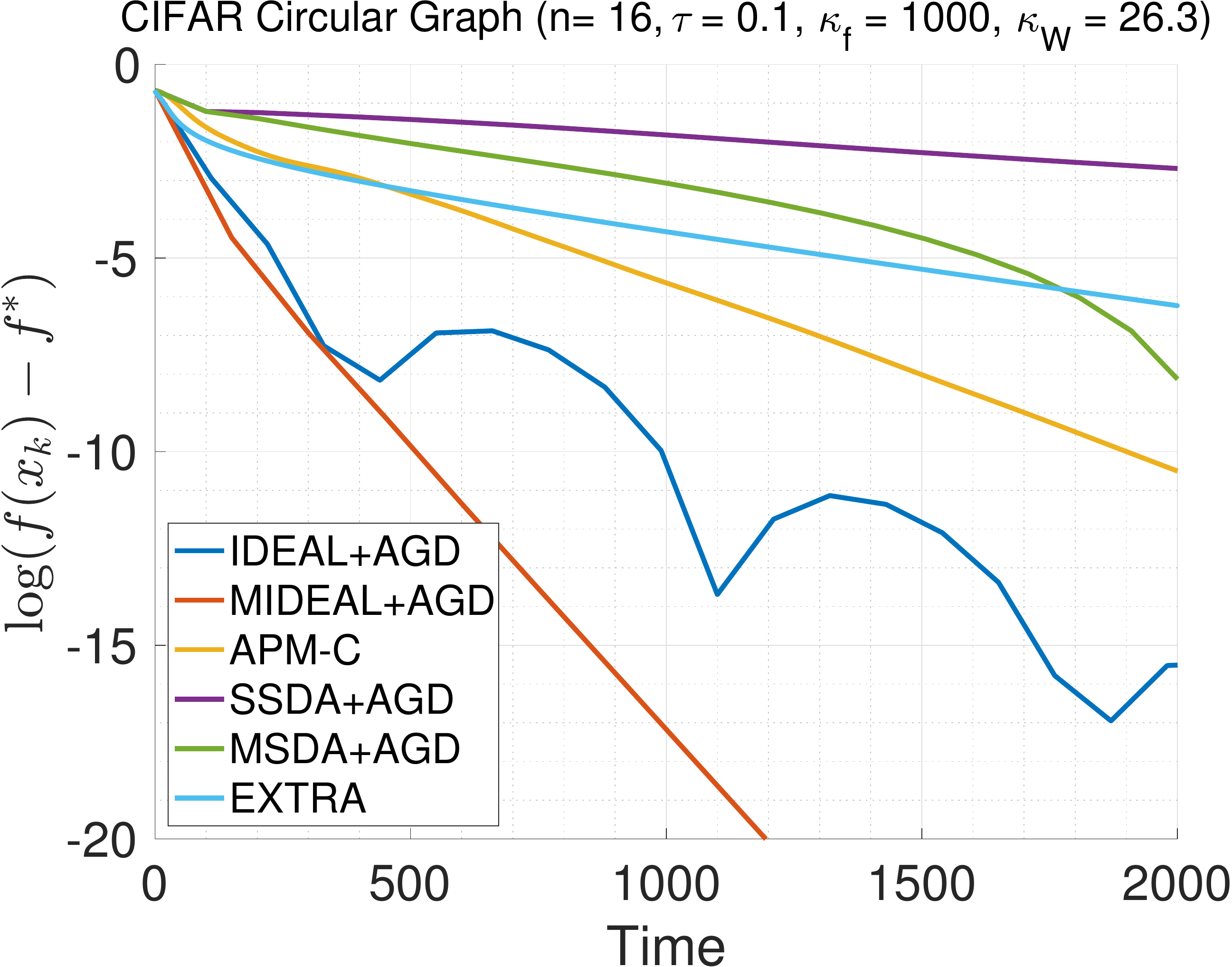}
\end{subfigure}%
\begin{subfigure}{.33\textwidth}
  \centering
    \includegraphics[width=\linewidth]{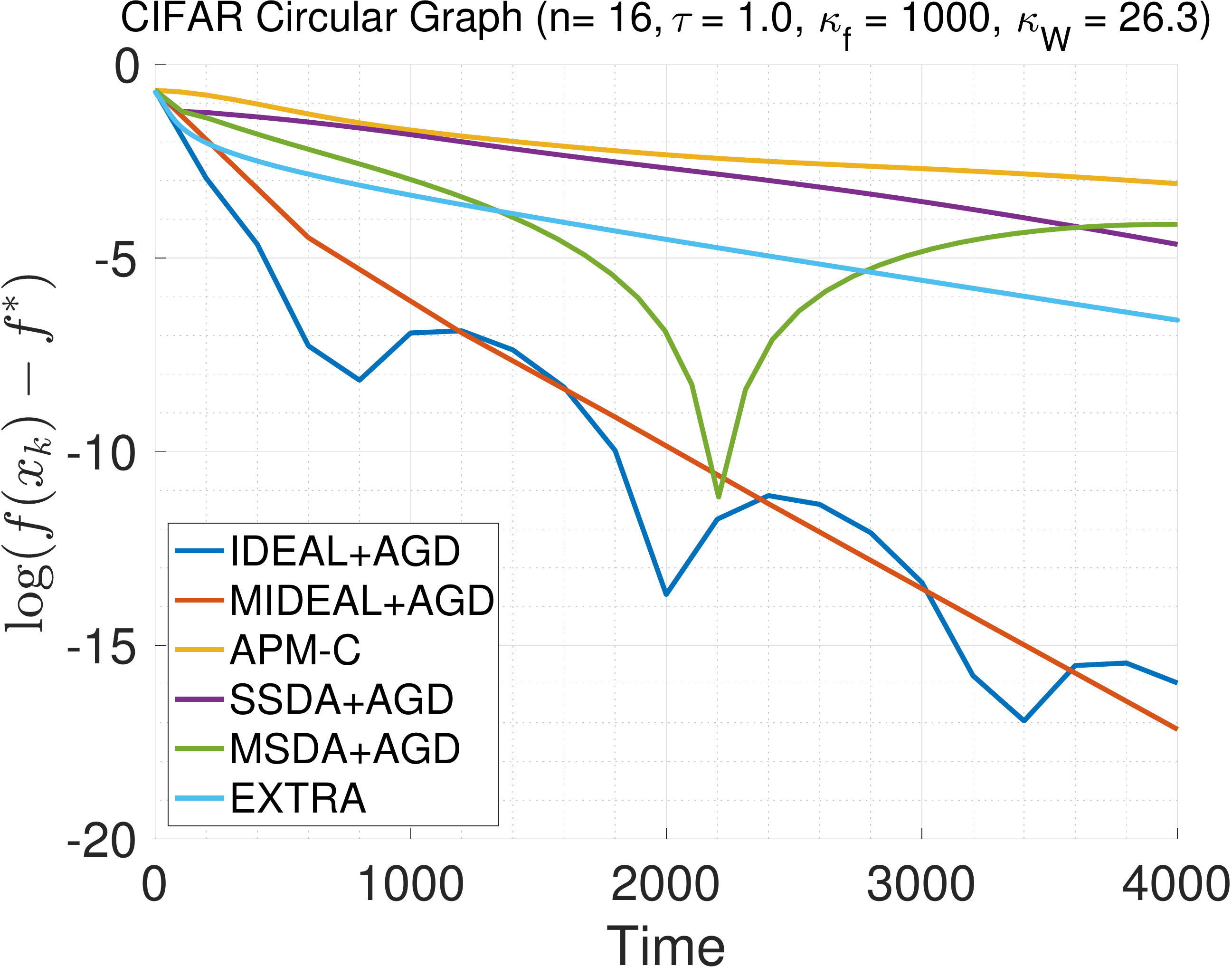}
\end{subfigure}
\begin{subfigure}{.33\textwidth}
  \centering
    \includegraphics[width=\linewidth]{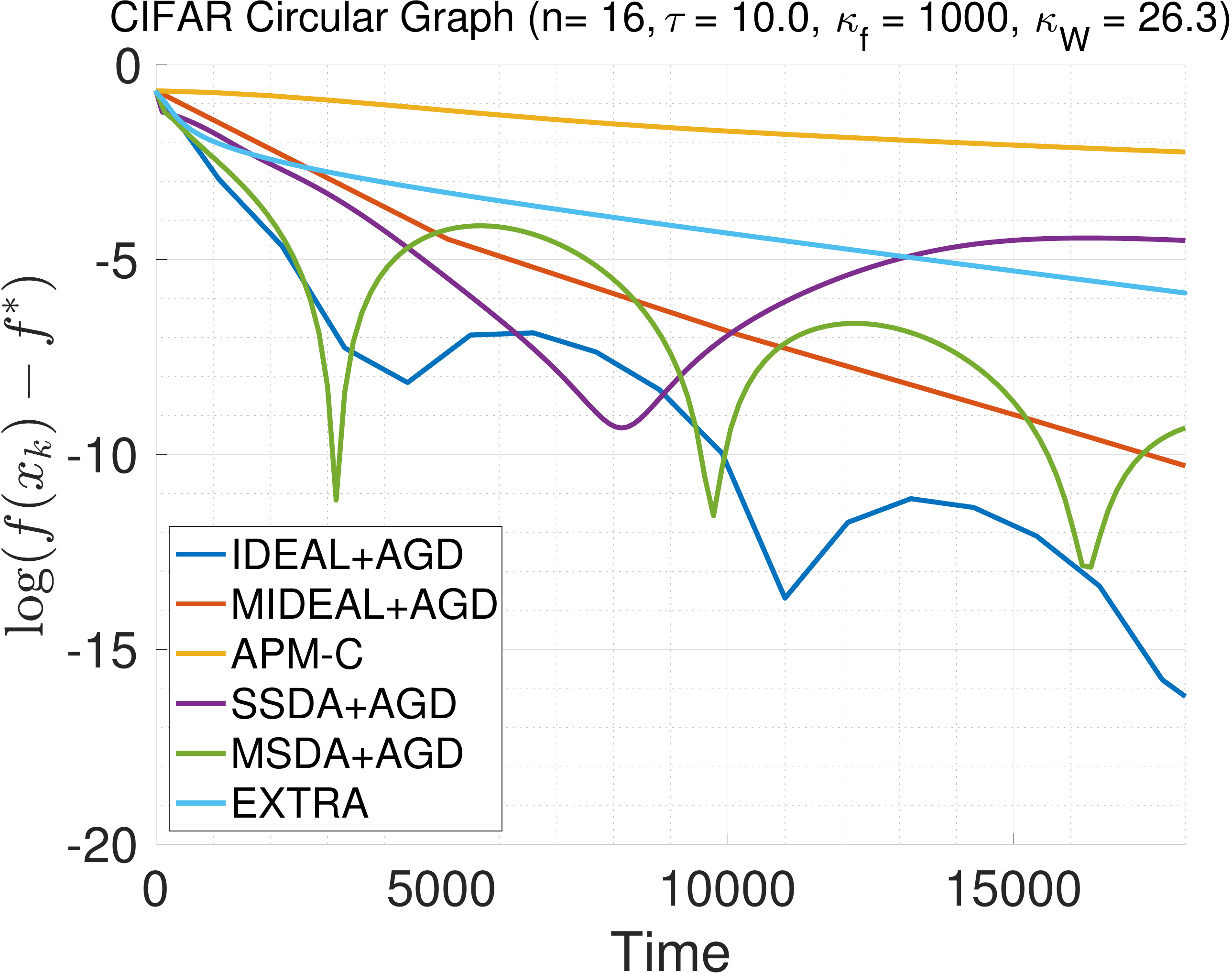}
\end{subfigure}%
\\
\begin{subfigure}{.33\textwidth}
  \centering
    \includegraphics[width=\linewidth]{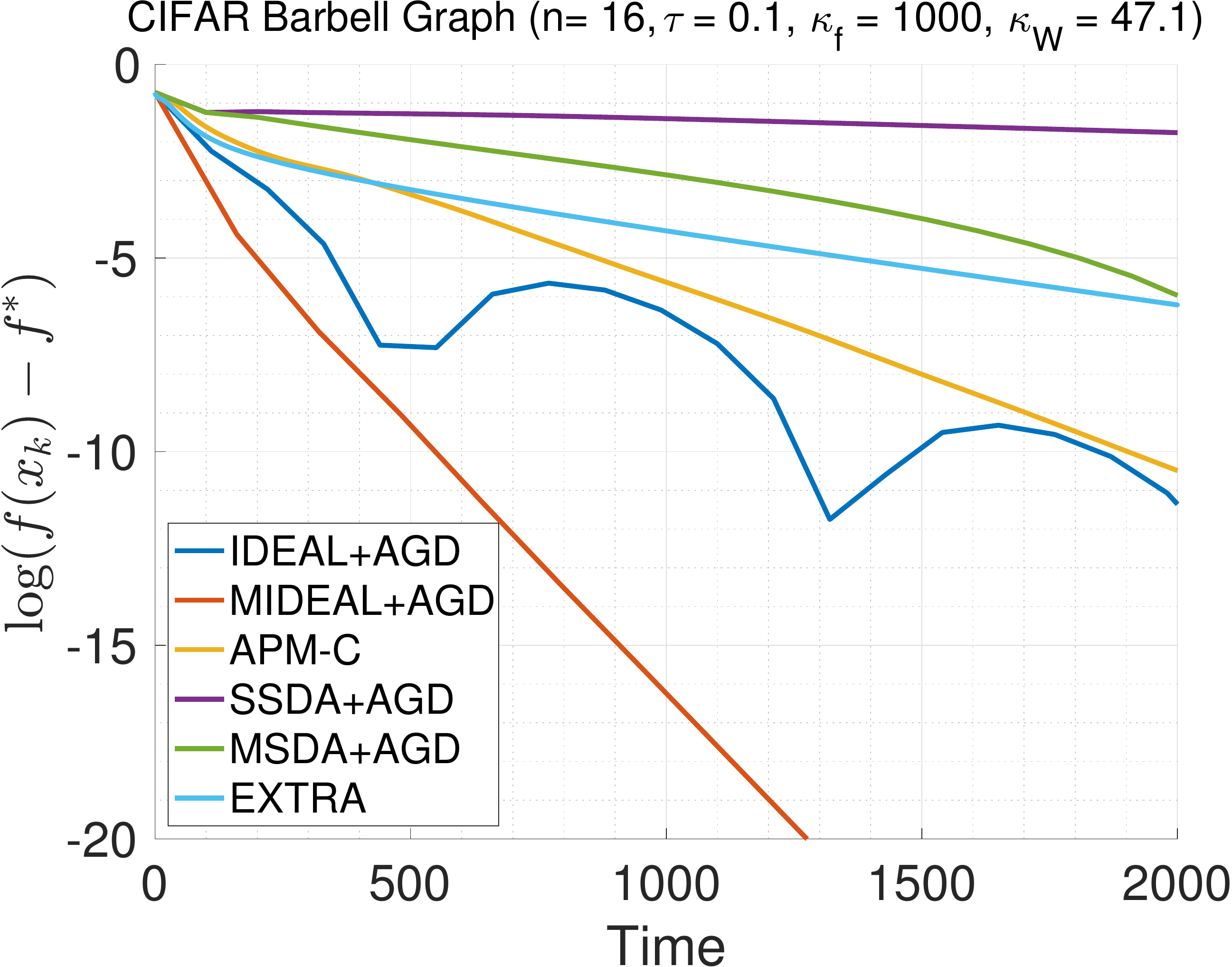}
\end{subfigure}%
\begin{subfigure}{.33\textwidth}
  \centering
    \includegraphics[width=\linewidth]{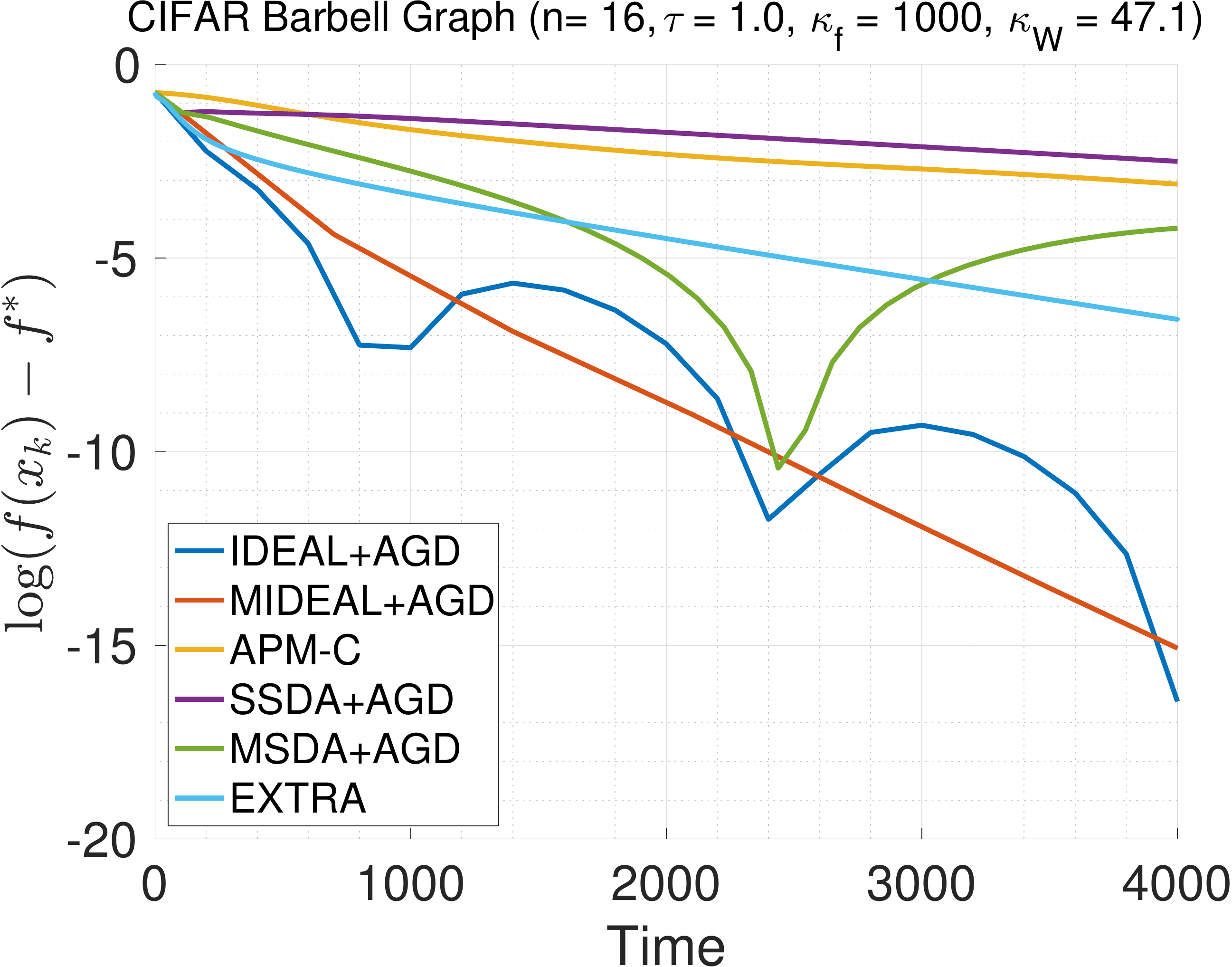}
\end{subfigure}%
\begin{subfigure}{.33\textwidth}
  \centering
    \includegraphics[width=\linewidth]{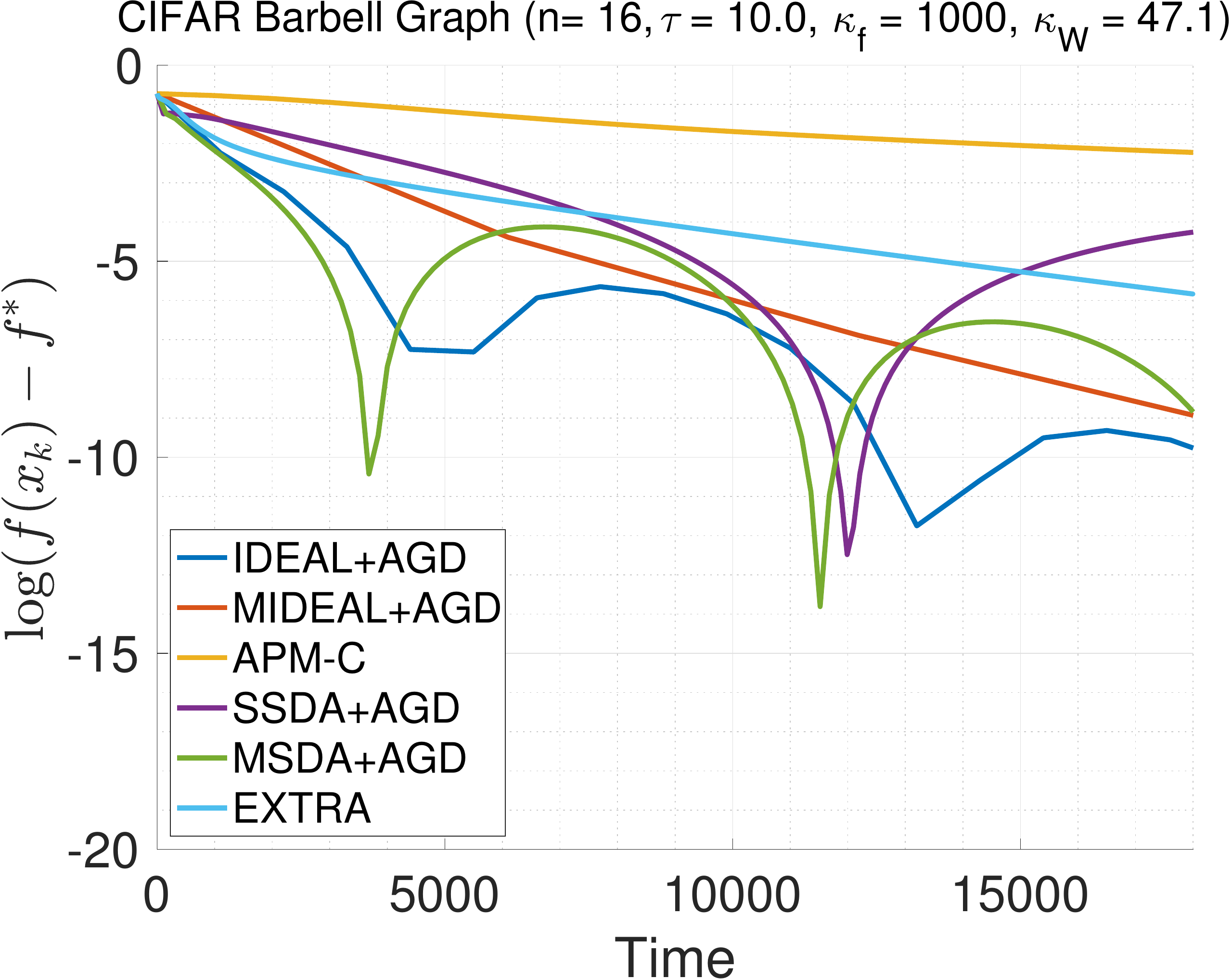}
\end{subfigure}%
\caption{{\bf CIFAR experiments}: we conduct experiments on two classes of CIFAR dataset, where the feature representation of each image was computed using an unsupervised convolutional kernel network Mairal~\cite{mairal2016end}. We observe similar phenomenon as in the MNIST experiment, that the multi-stage algorithm MIDEAL outperforms when the communication cost $\tau$ is low and the \our~outperforms in the other cases.}\label{fig:exps-cifar}
\end{figure}


\begin{figure}[t]
\begin{subfigure}{.45\textwidth}
  \centering
    \includegraphics[width=\linewidth]{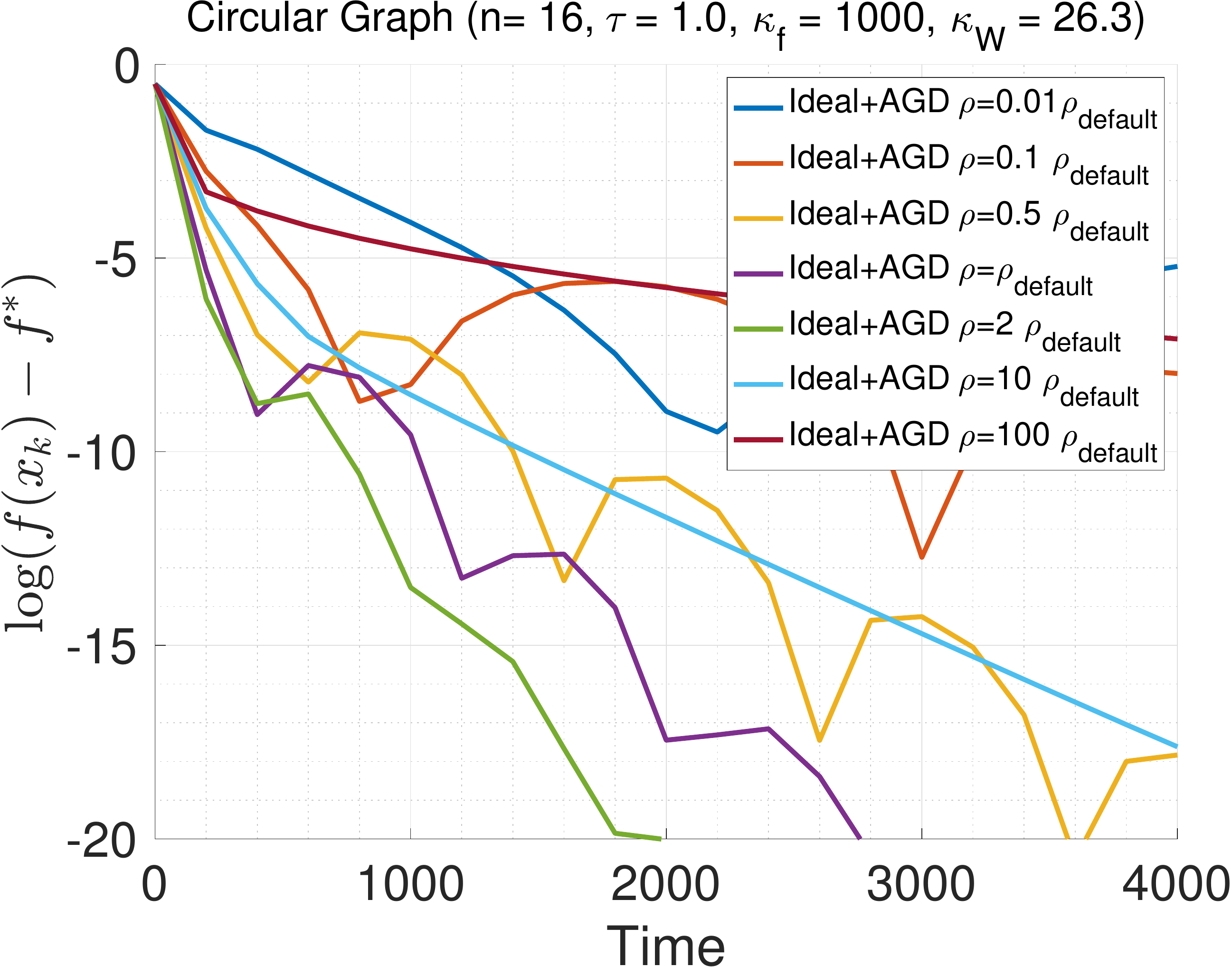}
\end{subfigure}%
\begin{subfigure}{.45\textwidth}
  \centering
    \includegraphics[width=\linewidth]{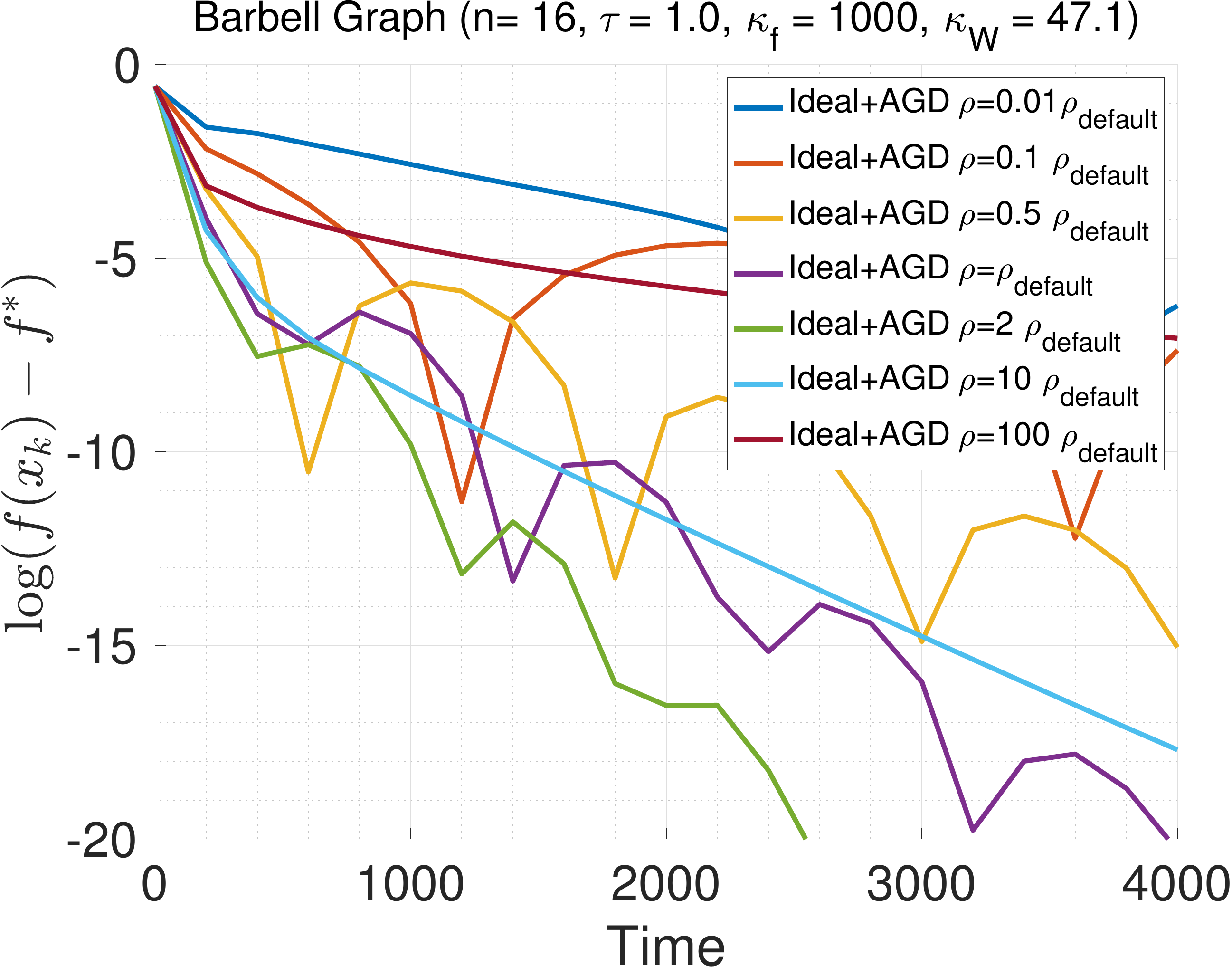}
\end{subfigure}

\caption{Ablation study on the {\bf  regularization parameter $\rho$} in \our~framework. For all the experiments, we use AGD as inner loop solver and set the same parameters as predicted by theory. We observe that when $\rho$ is selected in the range $[0.5 \rho_{\text{default}}, 10 \rho_{\text{default}}]$, the perfomance of the algorithm is quite similar and robust. We also observe that using a small $\rho$ degrades the performance of the algorithm, this phenomenon is consistent with the observation that the inexact SSD~\cite{scaman2017optimal} does not perform well since it uses $\rho=0$. Another observation is that with larger $\rho$, such as $\rho = 2\rho_{\text{default}}$ or $10\rho_{\text{default}}$, the algorithm is more stable with less zigzag oscillation, which is preferable in practice.}
\end{figure}

\begin{figure}[t]
\begin{subfigure}{.45\textwidth}
  \centering
    \includegraphics[width=\linewidth]{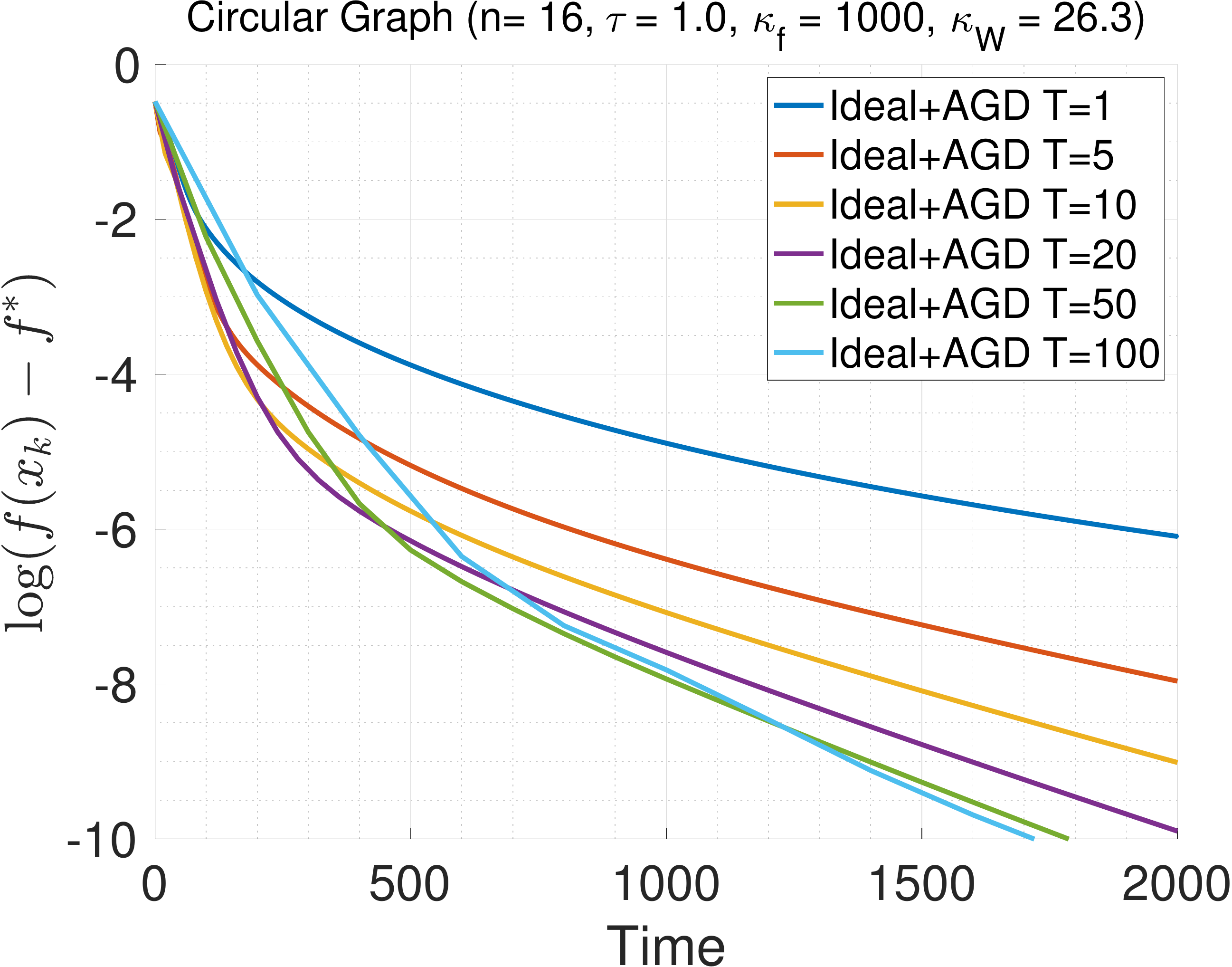}
\end{subfigure}%
\begin{subfigure}{.45\textwidth}
  \centering
    \includegraphics[width=\linewidth]{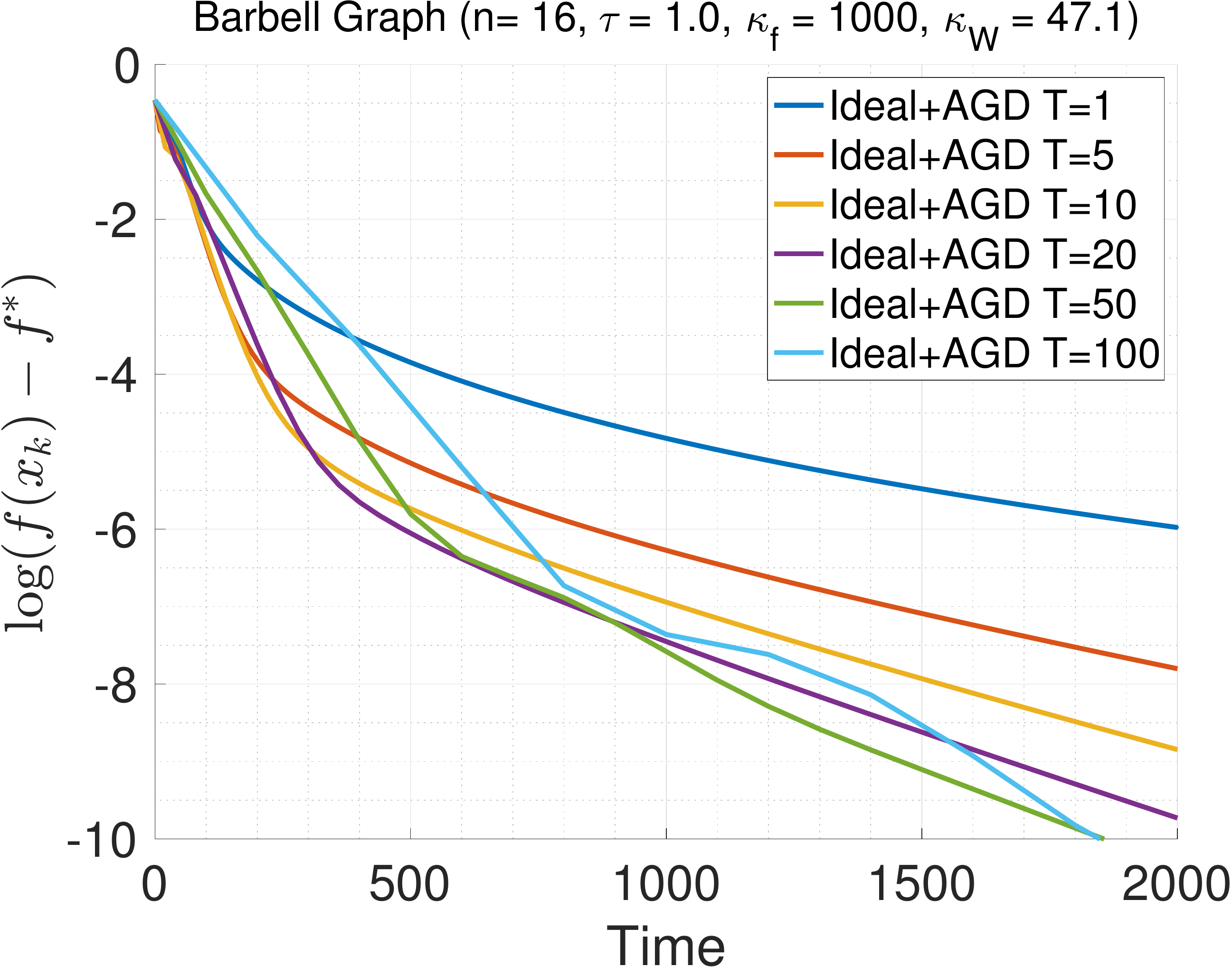}
\end{subfigure}

\caption{Ablation study on the {\bf inner loop complexity $T_k$} in \our~framework. When the inner loop iteration is small, the algorithm becomes less stable, so we have decreased the momentum parameters to ensure the convergence. For these experiments, we use AGD solver with $\beta_{in}=0.8$ and $\beta_{out} = 0.4$. As we can see, it is beneficial to perform multiple iterations in the inner loop rather than taking $T$=1 as in the EXTRA algorithm~\cite{shi2015extra}. }
\end{figure}

\end{document}